\documentclass[final,leqno,onefignum,onetabnum]{siamltex1213}

\usepackage{showkeys}
\usepackage{amssymb,amsmath}
\usepackage{algorithm}
\usepackage{algorithmic}

\newtheorem{remark}{\bf  Remark}[section]
\newtheorem{assumption}{\bf  Assumption}

\newcommand{\bbA}{{\mathbb{A}}}
\newcommand{\bbC}{{\mathbb{C}}}
\newcommand{\bbM}{{\mathbb{M}}}
\newcommand{\bbE}{{\mathbb{E}}}
\newcommand{\bbH}{{\mathbb{H}}}

\newcommand{\bbR}{{\mathbb{R}}}
\newcommand{\bbN}{{\mathbb{N}}}
\newcommand{\bbJ}{{\mathbb{J}}}
\newcommand{\bbP}{{\mathbb{P}}}

\newcommand{\cA}{\mathcal{A}}
\newcommand{\cB}{\mathcal{B}}
\newcommand{\cC}{\mathcal{C}}
\newcommand{\cD}{\mathcal{D}}

\newcommand{\cF}{\mathcal{F}}

\newcommand{\cJ}{\mathcal{J}}

\newcommand{\cL}{\mathcal{L}}
\newcommand{\cM}{\mathcal{M}}
\newcommand{\cN}{\mathcal{N}}

\newcommand{\cG}{\mathcal{G}}

\newcommand{\cQ}{\mathcal{Q}}
\newcommand{\cS}{\mathcal{S}}

\newcommand{\argmax}{\operatornamewithlimits{argmax}}
\newcommand{\argmin}{\operatornamewithlimits{argmin}}

\newcommand{\bsnull}{{\boldsymbol 0}}

\newcommand{\bszeta}{\boldsymbol{\zeta}}
\newcommand{\bsxi}{{\boldsymbol{\xi}}}
\newcommand{\bslambda}{{\boldsymbol{\lambda}}}
\newcommand{\bsrho}{{\boldsymbol{\rho}}}
\newcommand{\bstau}{\boldsymbol{\tau}}
\newcommand{\bseta}{\boldsymbol{\eta}}
\newcommand{\bspsi}{\boldsymbol{\psi}}

\newcommand{\bsnu}{{\boldsymbol{\nu}}}
\newcommand{\bsmu}{{\boldsymbol{\mu}}}

\newcommand{\bse}{{\boldsymbol{e}}}
\newcommand{\bsm}{{\boldsymbol{m}}}

\newcommand{\bsd}{\boldsymbol{d}}

\newcommand{\beq}{\begin{equation}}
\newcommand{\eeq}{\end{equation}}
\newcommand{\ba}{\begin{array}}
\newcommand{\ea}{\end{array}}

\newcommand{\pnote}[1]{\textcolor{black}{#1}}

\title{Hessian-based adaptive sparse quadrature for
  infinite-dimensional Bayesian inverse problems \thanks{This work was
    supported by DARPA's EQUiPS program under contract number
    W911NF-15-2-0121, NSF grants CBET-1508713 and ACI-1550593, and DOE
grants  DE-SC0010518 and DE-SC0009286.}}

\author{Peng Chen \thanks{Institute for Computational Engineering \&
    Sciences, The University of Texas at Austin, Austin, TX 78712
    (\email{peng@ices.utexas.edu, uvilla@ices.utexas.edu}).}  \and
  Umberto Villa \footnotemark[2] \and Omar Ghattas \thanks{Institute
    for Computational Engineering \& Sciences, Department of
    Mechanical Engineering, and Department of Geological Sciences, The
    University of Texas at Austin, Austin, TX 78712
    (\email{omar@ices.utexas.edu}).} } 
\begin{document}
\maketitle
\slugger{sisc}{xxxx}{xx}{x}{x--x}

\begin{abstract}
In this work we propose and analyze a Hessian-based adaptive sparse
quadrature to compute infinite-dimensional integrals with respect to
the posterior distribution in the context of Bayesian inverse problems
with Gaussian prior. Due to the concentration of the posterior
distribution in the domain of the prior distribution, a prior-based
parametrization and sparse quadrature may fail to capture the
posterior distribution and lead to erroneous evaluation results. By
using a parametrization based on the Hessian of the negative
log-posterior, the adaptive sparse quadrature can effectively allocate
the quadrature points according to the posterior distribution. A
dimension-independent convergence rate of the proposed method is
established under certain assumptions on the Gaussian prior and the
integrands. 
Dimension-independent and faster convergence than
$O(N^{-1/2})$ is demonstrated for a linear as well as a nonlinear
inverse problem whose posterior distribution can be effectively
approximated by a Gaussian distribution at the MAP point.

\end{abstract}



\pagestyle{myheadings}
\thispagestyle{plain}
\markboth{Hessian-based adaptive sparse quadrature for infinite-dimensional Bayesian inverse problems}{P. Chen, U. Villa and O. Ghattas}

\section{Introduction}

In many practical applications, a quantity of interest (QoI) predicted
by a mathematical model of a physical
system may be stochastic rather than deterministic  due to
uncertainties arising from inadequate
knowledge of input parameters that characterize 
material properties, computational geometries, initial and boundary
conditions, source terms, etc. Solution of an inverse problem will 
reduce these uncertainties based on available observational data on some
system outputs, leading to more
reliable computational predictions of the QoI. In a Bayesian
framework, the solution of the inverse problem entails the computation
of the posterior probability distribution of the uncertain input
parameter conditioned on (possibly) noisy observational data of the
system output. A prior probability distribution, based on experts'
belief, 
is prescribed for the uncertain
input parameter. Then, the posterior distribution of the uncertain
parameter can be formally obtained by the Radon--Nikodym derivative by
Bayes' theorem \cite{stuart2010inverse}. Once the posterior
distribution is known, one is often interested in evaluating the
statistical moments of the QoI (e.g., expectation and variance) with
respect to the posterior distribution for the assessment, design,
control and optimization of the system.

We consider the case of a spatially heterogeneous uncertain parameter,
i.e., a spatially correlated random field with a prescribed bounded
covariance between any two points of the physical domain. This setting
makes the uncertain parameter infinite-dimensional, which naturally
leads to an infinite-dimensional integration problem for the
evaluation of the statistical moments of the QoI with respect to the
posterior distribution. Several computational challenges are commonly
faced for such integration problems. Traditional deterministic
integration methods face the curse of dimensionality, i.e., the
computational complexity grows exponentially fast with respect to the
parameter dimension so that only a limited number of dimensions can be
resolved. On the other hand stochastic integration methods such as
Monte Carlo (MC) or Markov-chain Monte Carlo (MCMC) do not depend on
the parameter dimension but on the variance of the QoI; however the
convergence of the quadrature errors, $O(N^{-1/2})$ with $N$
quadrature samples \cite{hoang2013complexity}, is often very slow. In
addition, the posterior distribution, unlike the prior distribution,
is usually not explicitly available but rather is implicitly
represented by means of the Bayes Theorem and becomes rather
concentrated in the parameter space when the observational data are
very informative and thus is difficult to sample efficiently. Finally,
when the system is modeled by partial differential equations (PDE) in
complex geometries, evaluation of the QoI at each quadrature sample
require a full PDE solve that can involve expensive large-scale
computations, so that only a small number of samples can be computed.

We tackle these computational challenges by exploiting the
\emph{intrinsic sparsity} of the integration problem with respect to
the posterior distribution and the \emph{structure} of the inverse
problem. By sparsity, we mean that the QoI has markedly different
sensitivity in different parameter dimensions under the posterior
distribution; in other words, the dependence of the QoI on different
parameter dimensions is rather anisotropic with suitable
parametrization of the uncertain parameter, so that we can identify
the most sensitive dimensions and allocate most of the computational
effort in these dimensions \cite{schwab2012sparse, Schillings2013,
  schillings2014scaling, chen2015sparse, chen2016sparse,
  cui2016dimension}. By structure, we mean that high-order derivatives
of the parameter-to-observable map with respect to the parameter at
the maximum a posteriori (MAP) point, in particular
the Hessian, describe the inherent structure of the manifold of the
posterior distribution about the MAP point, which is especially useful
when the posterior distribution is relatively concentrated at the MAP
point due to informative observational data \cite{girolami2011riemann,
  martin2012stochastic, bui2013computational,
  petra2014computational}. To exploit these opportunities and to make
the solution of the infinite-dimensional Bayesian inverse problem
computationally feasible, here we develop a novel Hessian-based adaptive
sparse quadrature for integration of a QoI with respect to the
posterior distribution.

More specifically, we consider a Gaussian distribution in a Hilbert
space as the prior distribution for the uncertain parameter, where the
covariance operator is given by a fractional power of a second order
elliptic differential operator. We then compute the MAP point by
solving an optimization problem using a Lagrangian variational
approach and an inexact Newton--conjugate gradient solver. The covariance of the
Gaussian approximation of the posterior distribution at the MAP point,
to which we refer as the Gaussian posterior covariance, is given
as the inverse of the Hessian of the negative log-posterior
distribution. Like the prior covariance, this covariance is of trace
class; indeed we show that the eigenvalues of this covariance are
less than or equal to those of the prior covariance in every
dimension as long as the Hessian of the data misfit is non-negative.   
However, the Gaussian posterior covariance operator is formally a
large dense operator, and it is not available explicitly, but
implicitly through its action on a given function. To apply this
covariance operator to a function is very expensive, as it requires
solution of the so-called incremental forward and adjoint problems. To
make actions of the Gaussian posterior covariance computationally
feasible, we compute its eigendecomposition by a Lanczos
\cite{calvetti1994implicitly} or a randomized \cite{halko2011finding}
method. The fast decay of the eigenvalues of the Gaussian posterior
covariance allows us to express the infinite-dimensional uncertain
parameter by a truncated Karhunen--Lo\`eve expansion on the
eigenpairs, parametrized by a vector of independent and identically
distributed Gaussian random variables.
Our approach is in contrast with \cite{schillings2014scaling}, in
which the uncertain parameter is first represented by a truncated
Karhunen--Lo\`eve expansion on the eigenpairs of the prior covariance,
and then the coordinates corresponding to the Gaussian random
variables are transformed according to the Hessian of the negative
log-posterior with respect to these random variables. The advantage of
our approach is that we only need to compute the eigenpairs that are
important according to the Gaussian posterior covariance, which can
save considerable computational cost if a large number of the
eigenpairs of the prior covariance have to be kept in order to capture
the main uncertainty in the Gaussian posterior distribution. For
instance, an extremely large number of prior eigenpairs would have to
be used to capture the localized fine scales of the parameter field
for the Antarctic ice sheet flow inverse problem in
\cite{isaac2015scalable}, while many fewer eigenpairs of the Gaussian
posterior covariance are needed.
 
The sensitivity of the QoI with respect to different parameter
dimensions represented by the Gaussian random variables is
characterized by the fast decay of the eigenvalues of the Gaussian
posterior covariance. A dimension-adaptive sparse quadrature
\cite{gerstner2003dimension, chen2015new, Schillings2013} is then
constructed for the parametric integration problem. We establish under
certain assumptions on the decay of the eigenvalues of the Gaussian
prior and the regularity of the integrand that the convergence of the
adaptive sparse quadrature is bounded by
$\mathcal{O}(N^{-s})$, independent of the number of the parameter
dimensions. The exponent $s$ depends only on the algebraic decay
rate of the eigenvalues; the sparser the problem is, i.e., the
faster the eigenvalues decay, the faster the convergence becomes. In
particular, for $s$ larger than $1/2$, the deterministic sparse
quadrature achieves faster convergence than the stochastic MC or MCMC
quadrature. We demonstrate the concentration of the posterior
distribution in the region of the prior distribution, the efficacy of
the Hessian-based allocation of the quadrature points, and 
the convergence property of the sparse quadrature in both a linear and a
nonlinear inverse problem. 

The paper is organized as follows: Section \ref{sec:infdimBay}
presents the infinite-dimensional setting of Bayesian inverse
problems, computational tasks and challenges; the Hessian-based
parametrization is developed in Section \ref{sec:HessianPara};
Section \ref{sec:adaptivesparsequad} describes the adaptive
sparse quadrature and provides a dimension-independent convergence analysis;
Section \ref{sec:Numerical} is devoted to solution of two inverse
problems, one linear and the other nonlinear; and conclusions are
drawn in Section \ref{sec:conclusion}.  

\section{Infinite-dimensional Bayesian Inversion}
\label{sec:infdimBay}
\subsection{Overview}

Inverse problems consist of finding an unknown system input parameter given observational (or experimental) data on some system output. These problems are most often ill posed in the sense of Hadamard, i.e., violation of either existence or uniqueness, or the solution depends sensitively on the data. A classical remedy in a \emph{deterministic approach} is to solve \emph{a regularized least-squares minimization problem}. Alternatively, a {statistical approach}, formulated in the \emph{Bayesian framework}, seeks for the \emph{posterior probability distribution} of the parameter 
that maximizes the likelihood of the parameter to generate the observational data as system model output conditioned on its \emph{prior distribution} that incorporates available data or expert opinion/prior belief on the parameter. The latter approach is far richer than the former as it not only computes a \emph{maximum a posterior} (MAP) estimator as the most likely estimation of the parameter, but also quantifies the confidence of this estimator,  and provides the probability of a different estimator other than the MAP one. Moreover, different choices of the prior distribution make it versatile to model the available data or belief on the parameter. Furthermore, for a \emph{quantity of interest} (QoI) that depends on the parameter, possibly through the system state variable, more information can be assessed in the form of its statistics, such as mean, variance, and failure probability, which can be computed by either deterministic or statistical integration of the QoI with respect to the posterior distribution of the parameter. 

When the parameter is spatially correlated random field in the physical domain, the solution of the statistical inverse problems is the \emph{posterior measure} of the parameter in a function space, or more precisely a \emph{separable Banach space}, which is typically \emph{infinite-dimensional} as outlined in \cite{stuart2010inverse}. In this framework, an appropriate definition of the posterior measure is realized by the \emph{Radon--Nikodym} derivative with respect to the \emph{prior measure}, where the derivative is proportional to the \emph{likelihood function} based on \emph{conditional probability distribution}.  The prior measure of the parameter plays an important role to guarantee the well-posedness of the statistical inverse problems.
Moreover, it is the sparsity of the parameter encapsulated in the prior measure and successively inherited by the posterior measure that contributes to the development of sparse algorithms in sampling the parameter according to its posterior measure and computing the statistics of a given QoI. Among many choices of the prior measures, such as uniform, Besov, and Gaussian priors (see e.g. in \cite{dashti2013bayesian}), we consider the last one for its popularity in practical applications. 

In what follows, we present the Bayes' formula for the definition of the posterior measure in infinite dimensions in Section \ref{subsec:Bayes}, then in Section \ref{subsec:Stat} we address the central and computationally challenging task in Bayesian inverse problems, i.e. the computation of statistics of some QoI as an infinite-dimensional integration problem. Section \ref{subsec:Prior} is devoted to a practical construction of Gaussian priors.  

\subsection{Bayes' formula for infinite-dimensional inverse problems}
\label{subsec:Bayes}
By $X$ we denote a separable Banach space defined over the physical domain $D$, an open and bounded subset of $\bbR^d$ ($d = 1, 2, 3$) with Lipschitz boundary $\partial D$. 
By $m \in X$ we denote the uncertain system input parameter that belongs to $X$. We assume there exists a prior measure on $X$ for the parameter, denoted as $\mu_0$ such that $\mu_0(X) = 1$, i.e. functions randomly drawn from the prior measure lie in $X$ with probability 1. Moreover, we assume that the system output data $y$ live in a separable Banach space $Y$, which in most practical applications is a finite-dimensional Euclidean space $Y = \bbR^K$, for some $K \in \bbN$. By $\cG: X \to Y$ we denote the \emph{parameter-to-observable} map that maps the parameter $m$ to the observation data $y$, i.e. $y = \cG(m)$. Evaluations of $\cG$ involve the solution of the forward model, typically described by a partial differential equation (PDE). Furthermore, we assume that the data $y$ is corrupted by some addictive observation noise 
represented by a random variable $\eta \in Y$, i.e. 
\beq\label{eq:obsdata}
y = \cG(m)  + \eta\;.
\eeq
We assume the discrepancy $\eta$ to be Gaussian distributed with zero mean and covariance $\Gamma_{\text{noise}}$, i.e. $\eta \sim \cN(0,\Gamma_{\text{noise}})$, and we write the \emph{likelihood} (i.e. the probability density function of the observation data conditioned on the parameter, $y|m$) as
\beq\label{eq:potential}
\pi^y_{\text{like}}:= {\rm exp}\left( -\Phi(m,y) \right) = {\rm exp}\left(-\frac{1}{2} \|y - \cG(m)\|^2_{\Gamma_{\text{noise}}}\right),
\eeq
where $\Phi(m,y) := -\frac{1}{2} \|y - \cG(m)\|^2_{\Gamma_{\text{noise}}}$ denotes the \emph{negative log likelihood} or \emph{potential}. Here, we defined the $\Gamma_{\text{noise}}$ weighted norm as
\beq
 ||y - \cG(m)||^2_{\Gamma_{\text{noise}}} = \left( y - \cG(m) \right)^T \Gamma_{\text{noise}}^{-1} \left( y - \cG(m) \right).
\eeq 
The Bayes' formula establishes the relation between the posterior measure $\mu^y$ and the prior measure $\mu_0$ of the parameter through the Radon--Nikodym derivative 
\beq\label{eq:RadonNikodym}
\frac{d\mu^y}{d\mu_0} = \frac{1}{Z} \pi^y_{\text{like}},
\eeq
where the \emph{normalization} constant $Z$ is given by 
\beq
Z = \int_X \pi^y_{\text{like}} d\mu_0(m)\;. 
\eeq
A sufficient condition for the well-posedness of \eqref{eq:RadonNikodym} is the quadratic boundedness and the Lipschitz continuity of the map $\cG$ with respect to the parameter $m$, see \cite{stuart2010inverse}. Such conditions, which need to be verified for the specific PDE model at the hand, guarantee the absolute continuity of the posterior measure $\mu^y$ with respect to the prior measure $\mu_0$ and the positivity of the constant $Z$.

\subsection{Computation of statistics and challenges}
\label{subsec:Stat}

A central task of Bayesian inversion is to compute some statistics of a quantity of interest (QoI) $Q$ that depends on the parameter $m$, possibly through the solution of the PDE forward model $u$. For instance, we may be interested in computing the $k$-th moment of a functional $f$ of the solution $u$, i.e., $Q(m) = f^k(u(m))$, $k = 1, 2, \dots$. More precisely, for a given integrand $Q$, one needs to evaluate integrals with respect to the posterior measure of the form
\beq\label{eq:PostIntegral}
\bbE^{\mu^y}[Q] = \int_X Q(m) d\mu^y(m)\;.
\eeq
By the Bayes' formula \eqref{eq:RadonNikodym}, this is equivalent to evaluate the following integral with respect to the prior measure
\beq\label{eq:PriorIntegral}
\bbE^{\mu_0}[Q\pi^y_{\text{like}}]  =  \frac{1}{Z}\int_X Q(m) \exp\left(- \Phi(m,y)\right) d\mu_0(m).
\eeq

Evaluation of integrals \eqref{eq:PostIntegral} or \eqref{eq:PriorIntegral} is a notoriously computational challenge. First, the parameter is a random function in $X$ thus an infinite-dimensional integration needs to be performed, which confronts the \emph{curse-of-dimensionality} for most deterministic quadrature rules or \emph{slow convergence} for statistical quadrature rules. Second, each evaluation of $Q$ at a quadrature point (sample) involves the solution of the PDE forward problem, and therefore, for complex forward problems, standard algorithms for infinite-dimensional integration are computationally unfeasible due to the large number of evaluations of $Q$ required. In order to harness the total computational cost, we exploit the sparsity of the prior measure to design a novel sparse quadrature algorithm that breaks the curse-of-dimensionality and achieves fast convergence of the quadrature error.

\subsection{The Gaussian prior}
\label{subsec:Prior}
We first consider an abstract construction of the infinite-dimensional Gaussian prior through the pushforward of the probability measure $\bbP$ on the i.i.d. standard Gaussian random sequence $\bsxi = (\xi_j)_{j\geq 1}$  under a map that takes the sequence into the random function, as outlined in \cite{dashti2013bayesian}.
 In particular, we consider a real-valued separable Hilbert space $X$ defined over the physical domain $D$, which is equipped with the inner-product $\langle \cdot, \cdot \rangle$ and the induced norm $||\cdot||_{X}$. For such space $X$, we denote its orthonormal basis as $(\psi_j)_{j\geq 1}$. Then we can definite the map 
 \beq\label{eq:PushMap}
 m: \bbR^{\infty} \to X, \quad \text{ where } m(\bsxi) = m_0 + \sum_{j \geq 1} \sqrt{\lambda_j} \psi_j \xi_j , \quad \xi_j \stackrel{i.i.d.}{\sim} \cN(0,1)\;,
 \eeq
where the mean $m_0 \in X$, the sequence $\bslambda := (\lambda_j)_{j\geq 1} \in \ell^1(\bbR^\infty)$, and $\cN(0,1)$ represents the standard Gaussian distribution. We say that the map defined above is a random parameter in $X$ equipped with Gaussian measure. For a Gaussian prior measure $\mu_0$ defined through the pushforward map \eqref{eq:PushMap}, the integral \eqref{eq:PriorIntegral} can be equivalently written after the change of variables as 
\beq\label{eq:PostIntegralPara}
\bbE^{\mu_0}[Q\pi^y_{\text{like}}] = \frac{1}{Z}\int_\Xi Q(m(\bsxi)) \exp\left(- \Phi(m(\bsxi),y)\right) d\mu(\bsxi)\;,
\eeq 
with the normalization constant
\beq\label{eq:PriorIntegralPara}
 Z = \int_\Xi \exp\left(- \Phi(m(\bsxi),y)\right) d\mu(\bsxi)\;,
\eeq
where the integration domain becomes $\Xi = \bbR^\infty$, and the Gaussian measure $\mu$ on $\Xi$ is defined by the tensor product of the univariate Gaussian measure as
\beq 
\mu(\bsxi) = \bigotimes_{j\geq 1} \rho(\xi_j)d\xi_j, 
 \text{ where } \rho(\xi_j) = \frac{1}{\sqrt{2\pi}} e^{-\xi_j^2/2}\;.
\eeq
Thus the infinite-dimensional integration \eqref{eq:PriorIntegral} in the function space $X$ with measure $\mu_0$ becomes an infinite-dimensional parametric integration in $\Xi$ with measure $\mu$.

Before proceeding to the numerical integration with respect to the parametrized random function $m$, we have to make it precise for the construction of the orthonormal basis $(\psi_j)_{j \geq 1}$, which are functions in the physical domain $D$. For rectangular domain $D$, e.g. $(0,1)^d$,  tensor product of trigonometric functions or orthonormal polynomials such as Legendre polynomials could be used. For more general domain $D$ with arbitrary shape, we turn to the covariance operator of Gaussian prior for its practical construction, which, in the Hilbert space $X$, is given by \cite{stuart2010inverse}
\beq\label{eq:GaussPrior}
\cC_0 := \bbE^{\mu_0}[(m-m_0)\otimes (m-m_0)] \equiv  \bbE^{\mu}[(m-m_0)\otimes (m-m_0)] = \sum_{j\geq 1} \lambda_j \psi_j \otimes \psi_j\;,
\eeq
where we have used the parametric expression of $m$ in \eqref{eq:PushMap} and the i.i.d. property of $\xi$. By the orthonormality of the basis $(\psi_j)_{j\geq 1}$, it is ready to have
\beq\label{eq:priorKL}
\cC_0 \psi_k = \sum_{j\geq 1} \lambda_j \langle\psi_j, \psi_k\rangle  \psi_j = \lambda_k \psi_k, \quad k = 1, 2, \dots, 
\eeq
which implies that $(\lambda_j, \psi_j)_{j\geq 1}$ are the eigenpairs of the covariance operator $\cC_0$. Therefore, to obtain the parametric expression \eqref{eq:PushMap}, which is known as \emph{Karhunen--Lo\`eve expansion} of the random function $m$ with Gaussian measure $\cN(m_0, \cC_0)$, we only need to specify the mean $m_0$ and the covariance operator $\cC_0$. A commonly used covariance operator is the fractional powers of a `Laplacian-like' operator $\cA$, i.e. $\cC_0 = \cA^{-\alpha}$ for certain $\alpha > 0$, where $\cA$ is positive definite, self-adjoint, and invertible. 
A specific example of $\cA$, as considered in \cite{bui2013computational}, is the elliptic differential operator 
\beq\label{eq:elldiffoper}
\cA u  := - \beta \nabla \cdot( \Theta \nabla u) + \gamma u, \quad u \in H^1(D)\;,
\eeq 
where $H^1(D)$ is the conventional Hilbert space; $\beta , \gamma > 0$, which control the variance of the parameter, the larger $\beta$ and $\gamma$ are, the smaller the variance becomes; $\Theta \in \bbR_{s.p.d.}^{d\times d}$ is symmetric positive definite, which controls the anisotropic property of the parameter; the power $\alpha > d/2$ controls the asymptotic decay rate of the eigenvalues of $\cC_0$, which determines the \emph{sparsity} of the parametrization of the random function $m$ under the Gaussian prior measure $\cN(m_0, \cC_0)$. In practice, these coefficients can be calibrated based on direct available data on the parameter or expert opinions.

\section{Hessian-based Parametrization}
\label{sec:HessianPara}
Upon the parameterization of the random function as presented in the last section, evaluation of statistics for any given QoI in \eqref{eq:PostIntegral} and \eqref{eq:PriorIntegral} becomes the infinite-dimensional parametric integration in \eqref{eq:PostIntegralPara} and \eqref{eq:PriorIntegralPara}. However, direct numerical integration of these quantities brings great challenges not only because of the curse-of-dimensionality, but also due to the fact that the posterior measure might be rather different from the prior measure, for instance the mean of the posterior measure is driven far away from that of the prior by the observational data, or the posterior measure tends to be more concentrated in a small region than the prior measure. Therefore, the numerical quadrature points allocated either deterministically or statistically according to the prior measure might not efficiently capture the posterior measure, or in the parametric setting, the posterior `density' $\pi_{\text{like}}^y(\bsxi) \bsrho(\bsxi)$ (which is not a true density but zero in the infinite-dimensional setting, here we assume that $\xi$ is finitely supported) might be extremely small where the prior `density' $\bsrho(\bsxi)$ is relatively big, and vice versa. In order to tackle this challenge, we take the observation data into account and push the prior measure to the posterior measure or the one at least close to it accordingly by developing a Hessian-based parametrization method. We describe the MAP point and the Hessian of the cost functional in subsection \ref{subsec:MAP} and \ref{subsec:Hessian}, followed by the Hessian-based parametrization of the random function in subsection \ref{subsec:HessianPara}.

\subsection{The MAP point}
\label{subsec:MAP}
The MAP point characterizes the point in the parameter space where the posterior measure might center or concentrate at.
In the finite-dimensional setting, the MAP point is defined as the point where the probability density function (pdf) of the posterior distribution attains the maximum value. However, this definition does not directly generalize to the infinite-dimensional setting as the pdf of the posterior measure does not exist or it is zero at every point. Nevertheless, we can define the MAP point ${m_1} \in X$ as the one where the ball $B(m_c, \epsilon) = \{m \in X: ||m_c-m||_X \leq \epsilon\}$, centered at $m_c$ with radius $\epsilon$, attains the maximum probability given by $\mu^y(B(m_c, \epsilon))$ when $\epsilon \to 0$, i.e. 
\beq
m_1 := \lim_{\epsilon \to 0}  \argmax_{m \in X} \mu^y(B(m, \epsilon))\;.
\eeq
Under the assumption of Gaussian prior measure $\cN(m_0, \cC_0)$ with $\cC_0 = \cA^{-\alpha}$ defined in the last section, when the mean $m_0$ lies in the range of $\cC_0^{-1/2}$, or equivalently in the domain of the differential operator with fractional power $\cA^{\alpha/2}$, we have that ${m_1}$ also lies in this range. In fact, it is the \emph{Cameron--Martin} space denoted as $E$, which is a Hilbert space equipped with the inner-product $\langle \cdot, \cdot \rangle_{\cC_0} = \langle \cC_0^{-1/2} \cdot, \cC_0^{-1/2} \cdot \rangle$ and the induced norm $||m||^2_{\cC_0} = \langle m, m \rangle_{\cC_0}$. By the definition of the posterior measure in \eqref{eq:RadonNikodym}, we have that ${m_1}$ is the solution of the following minimization problem
\beq\label{eq:MAP}
\min_{m \in E} \cJ(m)
\eeq
where the \emph{cost functional} (a term corresponding to the regularized least-squares minimization problem in the deterministic inverse problems) $\cJ$ is defined as 
\beq\label{eq:cost}
\cJ(m) = \frac{1}{2} ||y - \cG(m) ||^2_{\Gamma_{\text{noise}}} + \frac{1}{2} ||m - m_0||^2_{\cC_0}\;,
\eeq
where the first term is the potential defined in \eqref{eq:potential} and the second term arises from the Gaussian prior assumption.
Under the quadratic boundedness and the Lipschitz continuity of the map $\cG$,  see in \cite{stuart2010inverse}, the optimization problem \eqref{eq:MAP} is well-defined and admit a unique solution. In the case where the parameter-to-observable map $\cG$ depends on the parameter through a PDE solution, which is of particular interest in many applications, we can solve the optimization problem by a Lagrangian variational approach, which bypasses the direct computation of the derivative of the map $\cG$ with respect to the parameter $m$. We will explain this approach in Section \ref{sec:Numerical} for specific PDE models. 

\subsection{The Hessian}
\label{subsec:Hessian}
The local curvature of $\cJ$ at the MAP point indicates the shape of the distribution of the parameter, which can be assessed by the Hessian of $\cJ$ evaluated at the MAP point, i.e.  
\beq\label{eq:HMAP}
H_{\text{MAP}} := D_m^2 \cJ(m)|_{m={m_1}} = H_{\text{misfit}} + \cC_0^{-1}\;,
\eeq
where we denote the Hessian of the misfit term--the potential \eqref{eq:potential}--evaluated at the MAP point as $H_{\text{misfit}}$, i.e. $H_{\text{misfit}} = D_m^2 \Phi(m)|_{m={m_1}}$. When the parameter-to-observable map $\cG$ is linear with respect to $m$, which we write as 
\beq\label{eq:linearInv}
\cG(m) = Gm\;,
\eeq 
where $G$ is a bounded linear map, i.e. $G \in \cL(X, Y)$, 
then the Hessian misfit becomes
\beq\label{eq:misfitHessian}
H_{\text{misfit}} = G^*\Gamma_{\text{noise}}^{-1} G\;,
\eeq
 where $G^*: Y \to X$ is the adjoint of $G$. When the map $\cG$ is nonlinear with respect to $m$, we can perform the linearization as an approximation of the $\cG$ as  
 \beq\label{eq:linearApprox}
 \cG(m) \approx \cG({m_1}) + \cG'({m_1})(m - {m_1})\;,
 \eeq
 where we have assumed that $\cG$ is Fr\'echet differentiable with respect to $m$ and by $\cG'$ we denote its Fr\'echet derivative. Under this linearization, the Hessian misfit can be obtained the same as in \eqref{eq:misfitHessian} with the linear map $G:=  \cG'({m_1})$. The linear approximation \eqref{eq:linearApprox} is reasonable when $\cG$ is nearly linear in the parameter space or at least in the region that the posterior measure is non-negligible. For instance when the observation noise is very small, or size of observation data is very big, the posterior measure decays very fast to zero for the parameter away from the MAP point \cite{dashti2013map}. In the case that the nonlinear map $\cG$ depends on the parameter through the PDE solution, we can also compute the Hessian misfit by a Lagrangian variational approach, which will be exemplified in Section \ref{sec:Numerical}.    
 
\subsection{Hessian-based parametrization}
\label{subsec:HessianPara}
For a linear parameter-to-observable map $\cG$, it is known (see \cite{stuart2010inverse}) that the covariance operator of the posterior measure $\cC_1$ is given by the inverse of the Hessian at the MAP point, i.e. 
\beq\label{eq:PostCovariance}
\cC_1 = H_{\text{MAP}}^{-1} = \left(H_{\text{misfit}} + \cC_0^{-1}\right)^{-1}\;,
\eeq 
which is self-adjoint, positive definite and of a trace-class compact operator. Consequently, the posterior measure is exactly the Gaussian measure $\mu_1 = \cN({m_1}, \cC_1)$, and the integration defined in \eqref{eq:PostIntegral} becomes  
\beq\label{eq:PostIntegralGauss}
\bbE^{\mu_1}[Q] := \int_X Q(m)d\mu_1(m)\;.
\eeq 
Let $(\lambda^{1}_j, \psi_j^{1})_{j \geq 1}$ denote the eigenpairs of the compact covariance operator $\cC_1$, then the parameter $m$ can be expressed by the Karhunen--Lo\`eve expansion as 
 \beq\label{eq:postm1}
 m^1(\bsxi) = m_1 + \sum_{j \geq 1} \sqrt{\lambda^{1}_j} \psi_j^{1} \xi_j, \quad \xi_j \stackrel{i.i.d.}{\sim} \cN(0,1)\;,
 \eeq
 Therefore, the integration defined in \eqref{eq:PostIntegralGauss} now becomes a parametric integration
 \beq\label{eq:PostInteGaussPara}
 \bbE^{\mu_1}[Q] = \int_{\Xi} Q(m^{1}(\bsxi)) d\mu(\bsxi)\;.
 \eeq 
We remark that the parametrization \eqref{eq:postm1} is obtained by the Karhunen--Lo\`eve expansion of the posterior covariance operator at the MAP point. A re-parametric formulation based on a transformation of the parametrization w.r.t. the prior covariance operator is presented in Appendix A, which is equivalent to our development above in the abstract functional space. However, they are rather different from the computational point of view, see the comparison in Appendix A. We also present an efficient computation of the Hessian-based parametrization in Appendix B.

For a nearly linear map $\cG$, we can also use this Gaussian measure as an approximation of the posterior measure and compute the integration \eqref{eq:PostIntegralGauss} to approximate the true integration \eqref{eq:PostIntegral}. For a nonlinear map $\cG$, we can compute the true integration \eqref{eq:PostIntegral} by changing the measure as used in importance sampling \cite{agapiou2015importance}
 \beq\label{eq:PostIntegraldu0du1}
 \bbE^{\mu^y}[Q]  = \frac{1}{Z}\int_X Q(m) \pi^{y,1}_{\text{like}}(m)d\mu_1(m), \text{ where } Z: = \int_X  \pi^{y,1}_{\text{like}}(m)d\mu_1(m)\;,
 \eeq 
 where the likelihood function $ \pi^{y,1}_{\text{like}}(m)$ is given by 
 \beq\label{eq:likelihood}
 \pi^{y,1}_{\text{like}}(m) =  \exp\left(\cJ({m_1})\right)\exp\left(-\Phi(m,y)\right) \frac{d\mu_0(m)}{d\mu_1(m)} =: \exp\left(-\cJ_1(m)\right)\;,
 \eeq
 where we have multiplied a rescaling constant $\exp\left(\cJ(m_1)\right)$ and defined
 \beq\label{eq:J1}
 \cJ_1(m) = \cJ(m) - \cJ(m_1) - \frac{1}{2}||m-m_1||_{\cC_1}^2\;, 
 \eeq
 where the cost functional $\cJ$ is defined in \eqref{eq:cost}. Note that, by construction, $\cJ_1(m)$ and its first and second derivative vanish for $m = m_1$, since $D_m \cJ|_{m = m_1} = 0$ and $D_{mm} \cJ|_{m = m_1} = {\cC_1}^{-1}$.
Under the parametrization \eqref{eq:postm1}, 
 the integrals \eqref{eq:PostIntegraldu0du1} become 
 \beq\label{eq:PostIntePara}
  \bbE^{\mu^y}[Q]  = \frac{1}{Z}\int_\Xi Q(m^1(\bsxi)) \pi^{y,1}_{\text{like}}(m^1(\bsxi))d\mu(\bsxi) \text{ and }  Z = \int_\Xi \pi^{y,1}_{\text{like}}(m^1(\bsxi))d\mu(\bsxi).
 \eeq

\section{Adaptive Sparse Quadrature for Bayesian Inversion}
\label{sec:adaptivesparsequad}
In this section, we present a sparse quadrature method for the computation of the parametric integration \eqref{eq:PostInteGaussPara} and \eqref{eq:PostIntePara} in the Bayesian inverse problems, following our work in \cite{chen2016adaptive}. We establish a dimension-independent convergence result of the sparse quadrature and present two algorithms for the construction of the sparse quadrature.
\subsection{Sparse quadrature}
For any integrable function $g: \Xi \to \bbR$, we define a sparse quadrature for the approximation of the integral $\bbE^\mu[g] = \int_\Xi g(\bsxi) d\mu(\bsxi)$ as 
\beq\label{eq:sparseQuad}
\cQ_\Lambda(g) = \sum_{\bsnu \in \Lambda} \triangle_\bsnu(g)\;.
\eeq
Here, $\bsnu = (\nu_1, \nu_2, \dots ) \in \cF$ is a multi-index in $\cF = \{\bsnu \in \bbN^\infty: |\bbJ_\bsnu| < \infty \}$, where $\bbJ_\bsnu$ is the support of $\bsnu$, i.e., $\bbJ_\bsnu = \{j \in \bbN: \nu_j > 0\}$. $\Lambda$ is an admissible index set satisfying
\beq
\text{for any } \bsnu \in \cF, \text{ if } \bsnu \in \Lambda, \text{ then } \bsmu \in \Lambda \text{ for all } \bsmu \preceq\bsnu \;,
\eeq 
 where $\bsmu \preceq\bsnu $ means $\mu_j \leq \nu_j \text{ for all } j \geq 1$. $\triangle_\bsnu$ is a tensor product difference quadrature operator defined by tensorization of a sequence of univariate difference operators  
 \beq\label{eq:TensorDiff}
 \triangle_\bsnu = \bigotimes_{j \geq 1} \triangle_{\nu_j} (g)\;,
 \eeq
 where the univariate difference quadrature operator is defined as 
 \beq
\triangle_{\nu_j} (g) = \cQ_{\nu_j}(g) - \cQ_{\nu_j-1}(g), \quad \nu_j = 0, 1, \dots\;.
 \eeq
Here the univariate quadrature operator $\cQ_{-1}(g) = 0$ and $\cQ_{\nu_j}(g)$ is given by 
 \beq
 \cQ_{\nu_j}(g) = \sum_{k_j = 0}^{m_{\nu_j}-1} w_j^{(k_j)} g(\xi_1, \dots, \xi_{j-1}, \xi_j^{(k_j)}, \xi_{j+1}, \dots), \quad \nu_j = 0, 1, \dots\;,
 \eeq
 where $(\xi_j^{(k_j)}, w_j^{(k_j)})$, $k_j = 0, \dots, m_{\nu_j}-1$, are the quadrature points and weights of certain quadrature rules, such as Gauss--Hermite and Genz--Keister, see \cite{chen2016adaptive} for details.

\subsection{Dimension-independent convergence rates}

In this section, we analyze the convergence of the quadrature error $|\bbE^{\mu}[g] - \cQ_{\Lambda_N}(g)|$ with respect to $N$, the cardinality of the admissible index set $\Lambda_N$, i.e., $N = |\Lambda_N|$, for $g$ in the context of the Bayesian inverse problems. 

We first present a regularity assumption of the parametric function $g: \Xi \to \bbR$ with respect to the parameter $\bsxi$, which is  \cite[Assumption 2]{chen2016adaptive}.

\begin{assumption}\label{ass:DeriBound}
Let $0< q < 1$, and $(\tau_j)_{j\geq 1}$ be a positive sequence such that 
\beq
(\tau_j^{-1})_{j\geq 1} \in \ell^q(\bbN)\;.
\eeq 
Let $r$ be the smallest integer such that $r > 14/q$,  we assume that the parametric function $g:\Xi \to \bbR$ satisfies 
\beq\label{eq:regularitytau}
\sum_{|\bsmu|_\infty \leq r} \frac{\bstau^{2\bsmu}}{\bsmu!} \int_{\Xi} |\partial^\bsmu_\bsxi g(\bsxi)|^2 d\mu(\bsxi) < \infty\;,
\eeq 
where $|\bsmu|_\infty = \max_{j\geq 1} \mu_j$, $\bstau^{2\bsmu}= \prod_{j\geq 1} \tau_j^{2\mu_j}$, $\bsmu! = \prod_{j\geq 1}\mu_j!$,  and $\partial_\bsxi^\bsmu g = \prod_{j\geq 1}\partial_{\xi_j}^{\mu_j} g$.
\end{assumption}


The following theorem from \cite{chen2016adaptive} provides sufficient condition for $g$ such that the sparse quadrature error converges with dimension-independent convergence rate. \pnote{Before stating the theorem, we introduce a sequence of useful constants 
\beq\label{eq:bnu}
b_\bsnu = \sum_{|\bsmu|_\infty\leq r} 
\left(
\begin{array}{cc}
\bsnu \\
\bsmu
\end{array}
\right) \bstau^{2\bsmu}, \quad \text{ with } 
\left(
\begin{array}{cc}
\bsnu \\
\bsmu
\end{array}
\right) = 
\prod_{j \geq 1}
\left(
\begin{array}{cc}
\nu_j \\
\mu_j
\end{array}
\right), \quad \bsnu \in \cF\;,
\eeq
which appear in the equivalence of \eqref{eq:regularitytau} and a weighted sum of the coefficients of the Hermite expansion of $g$, see \cite[Proposition 3.3]{chen2016adaptive}, or \cite[Theorem 3.1]{bachmayr2015sparse}.}

\begin{theorem}\label{thm:N-termConv} \cite[Theorem 3.5]{chen2016adaptive} Under Assumption \ref{ass:DeriBound},
there exists an admissible index set $\Lambda_N \subset \cF$, with the $N$ indexes $\bsnu$ corresponding to the smallest $b_\bsnu$ given by \eqref{eq:bnu},
and there exists a constant $C$ independent of $N$,
such that the error of the sparse quadrature \eqref{eq:sparseQuad} based on the Gauss--Hermite quadrature rule is bounded by
\beq
|\bbE^\mu[g]- \cQ_{\Lambda_N}(g)| \leq C (N+1)^{-s}, \quad s = \frac{1}{q} - \frac{1}{2}\;.
\eeq 
\end{theorem}
\begin{remark}\label{rmk:convRate}
The convergence rate in Theorem \ref{thm:N-termConv} does not depend on the parameter dimensions but only on the sparsity parameter $q$ that controls the regularity of $g$. We remark that this convergence rate is not necessarily optimal. In fact, $s = 1/q + 1/2$ is observed in numerics \cite{chen2016adaptive}. Note also that as the number of quadrature points \pnote{$N_p$} in $\Lambda_N$ grows as $O(N^2)$, see \cite[Proposition 18]{ernst2016convergence}, so that the convergence with respect to the number of quadrature points is deteriorated to $N_p^{-s/2}$. However, similar convergence rate $N_p^{-s}$ is observed in practice because the Gauss--Hermite quadrature $\cQ_\nu$ with $m_\nu = \nu+1$ points is exact for polynomials of degree $2\nu+1$, \pnote{see \cite{chen2016adaptive}.}
\end{remark}

The dimension-independent convergence result in Theorem \ref{thm:N-termConv} applies once Assumption \ref{ass:DeriBound} can be verified, see \cite{chen2016adaptive} for examples of parametric functions and parametric PDEs which satisfy this assumption. Here, we verify Assumption \ref{ass:DeriBound} for the parametric QoI $Q$ integrated with respect to the parameter $m^1$ in \eqref{eq:PostInteGaussPara}, which is parametrized according to the posterior measure.  For the prior measure $\cN(m_0, \cC_{0})$, where the covariance operator is constructed as $\cC_0 = \cA^{-\alpha}$, we assume that $\cA$ is a `Laplacian-like' operator satisfying the following conditions as in \cite{stuart2010inverse}.
\begin{assumption}
\label{ass:LaplaceOperator}
$\cA$ is defined on a Hilbert space $X$ over $D\subset \bbR^d$ and fulfills 
\begin{enumerate}
\item[A.1] $\cA$ is positive definite, self-adjoint, and invertible with eigenpairs $\{(\lambda_j, \psi_j)\}_{j \geq 1}$, where the eigenfunctions form an orthonormal basis of $X$.

\item[A.2] $\lambda = (\lambda_j)_{j\geq 1}$ grows asymptotically as $ j^{2/d}$, or $\lambda^{-1} \in \ell^q(\bbR^\infty)$, $\forall q > d/2$.

\end{enumerate}
\end{assumption}
As the prior covariance $\cC_0 = \cA^{-\alpha}$, the larger $\alpha$ is, the more sparse or anisotropic the parameter $m$ becomes with respect to different dimensions of the random variables $\bsxi = (\xi_j)_{j\geq 1}$ in the Karhunen--Lo\`eve expansion \eqref{eq:compactm}. Here the sparsity is indicated by the decay of the eigenvalues $\lambda^0 = (\lambda_j^0)_{j \geq 1} \in \ell^q(\bbR^\infty)$ of the prior covariance $\cC_0$, where we have $q > d/(2\alpha)$ according to assumption A.2. We expect that this sparsity is inherited in the Gaussian posterior measure, i.e., the eigenvalues $\lambda^1 = (\lambda^1_j)_{j\geq 1}$ of the Gaussian posterior covariance $\cC_1$ is such that $ \lambda^1 \in \ell^q(\bbR^\infty)$. This is established in the following lemma.
\begin{lemma}\label{lem:uncertaintyReduction}
Let $(\lambda_j^0)_{j \geq 1}$ and $(\lambda^1_j)_{j\geq 1}$ denote the non-increasing sequence of eigenvalues of the Gaussian prior covariance operator $\cC_0 = \cA^{-\alpha}$ with $\alpha > d/2$ in Assumption \ref{ass:LaplaceOperator} and the Gaussian posterior covariance operator $\cC_1$ in \eqref{eq:PostCovariance}, respectively, if the Hessian misfit $H_{\text{misfit}}$ in \eqref{eq:HMAP} is positive semi-definite, we have
\begin{equation}\label{eq:lambdacomparison}
0 \leq \lambda_j^1 \leq \lambda_j^0, \quad \forall j \geq 1\;.
\end{equation}
Moreover, under Assumption \ref{ass:LaplaceOperator}, we have $\lambda^1 = (\lambda_j^1)_{j\geq 1} \in \ell^q(\bbR^\infty)$, $\forall q > d/(2\alpha)$. 
\end{lemma}
\begin{proof}
For a positive semi-definite Hessian misfit $H_{\text{misfit}}$, we have 
\beq
\langle v, (H_{\text{misfit}} + \cC_0^{-1}) v \rangle \geq \langle v,  \cC_0^{-1} v \rangle \geq \lambda_1^\alpha > 0, \quad \forall v \in X \text{ and } ||v||_X = 1\;,
\eeq
where $\lambda_1$ is the smallest eigenvalue of $\cA$ in Assumption \ref{ass:LaplaceOperator}. Therefore, we have 
\beq
0 \leq \langle v, (H_{\text{misfit}} + \cC_0^{-1})^{-1} v \rangle \leq \lambda_1^{-\alpha}, \quad \forall v \in X \text{ and } ||v||_X = 1\;,
\eeq
which implies, by definition of $\cC_1$ in \eqref{eq:PostCovariance}, that $0 \leq \lambda_j^1 \leq \lambda_1^{-\alpha}$. Moreover, we can write 
\begin{equation}\label{eq:cC1}
\cC_1 = \cC_0^{1/2} (\cC_0^{1/2}H_{\text{misfit}} \cC_0^{1/2} + I)^{-1} \cC_0^{1/2}.
\end{equation}
As the Hessian misfit $H_{\text{misfit}}$ is self-adjoint (by definition) and positive semi-definite (by assumption), so is $\cC_0^{1/2}H_{\text{misfit}} \cC_0^{1/2}$ by Assumption \ref{ass:LaplaceOperator}. Hence, we can write 
\beq\label{eq:cC0H}
\cC_0^{1/2}H_{\text{misfit}} \cC_0^{1/2} = \sum_{j\geq 1} \tilde{\lambda}_j \tilde{\psi}_j \otimes \tilde{\psi}_j,
\eeq
where $(\tilde{\lambda}_j, \tilde{\psi}_j )_{j\geq 1}$ are the eigenpairs of $\cC_0^{1/2}H_{\text{misfit}} \cC_0^{1/2}$ with $\tilde{\lambda}_j \geq 0$, $\forall j \geq 1$. 
A combination of \eqref{eq:cC1} and \eqref{eq:cC0H} yields 
\beq\label{eq:cC1cC0}
\cC_1 = \cC_0 - \cC_0^{1/2} \cD \cC_0^{1/2}, \quad \text{ where } \cD = \sum_{j\geq 1} \frac{\tilde{\lambda}_j}{1+\tilde{\lambda}_j} \tilde{\psi}_j \otimes \tilde{\psi}_j.
\eeq
Since $\cC_0$ and $\cC_0^{1/2} \cD \cC_0^{1/2}$ are both compact and self-adjoint, so is $\cC_1$. By Courant--Fischer (min-max) theorem \cite{reed1978methods}, 
the eigenvalues of the prior covariance operator $\cC_0$ and the Gaussian posterior covariance operator $\cC_1$ can be written as 
\begin{equation}
\lambda_j^i := \sup_{\substack{V_j \in X \\ \text{dim}(V_j) = j}} \inf_{\substack{v \in V_j \\ ||v||_{X} = 1}} \langle v, \cC_i v \rangle, \quad j\geq 1, i = 0, 1,
\end{equation}
so that by \eqref{eq:cC1cC0} we have 
\beq\label{eq:lambda01}
\begin{split}
\lambda_j^1 & = \sup_{\substack{V_j \in X \\ \text{dim}(V_j) = j}} \inf_{\substack{v \in V_j \\ ||v||_{X} = 1}} \langle v, (\cC_0 - \cC_0^{1/2} \cD \cC_0^{1/2})v \rangle \\
& \leq \sup_{\substack{V_j \in X \\ \text{dim}(V_j) = j}} \inf_{\substack{v \in V_j \\ ||v||_{X} = 1}} \left( \langle v, \cC_0 v \rangle - \lambda_{\text{min}}\left(\cC_0^{1/2} \cD \cC_0^{1/2}\right) \right)
\\
& \leq \sup_{\substack{V_j \in X \\ \text{dim}(V_j) = j}} \inf_{\substack{v \in V_j \\ ||v||_{X} = 1}} \langle v, \cC_0 v \rangle = \lambda_j^0, \quad j \geq 1,
\end{split}
\eeq
where the minimum eigenvalue of $\cC_0^{1/2} \cD \cC_0^{1/2}$ satisfies $ \lambda_{\text{min}}\left(\cC_0^{1/2} \cD \cC_0^{1/2}\right) \geq 0$ since both $\cC_0$ in \eqref{eq:GaussPrior} and $\cD$ in \eqref{eq:cC1cC0} are positive. Therefore, \eqref{eq:lambdacomparison} is established. Under Assumption \ref{ass:LaplaceOperator}, we have $\lambda^0 = (\lambda_j^0)_{j\geq 1} \in \ell^q(\bbR^\infty) $, $\forall q > d/(2\alpha)$, so that by \eqref{eq:lambdacomparison} we obtain $\lambda^1 = (\lambda_j^1)_{j\geq 1} \in \ell^q(\bbR^\infty)$, $\forall q > d/(2\alpha)$.
\end{proof}

\begin{corollary}
For linear (and linearized) inverse problems, where the Hessian misfit is given by \eqref{eq:misfitHessian}, which is positive semi-definite, \eqref{eq:lambdacomparison} holds.
\end{corollary}

Based on the above analysis for the decay of the eigenvalues of the Gaussian posteriori covariance, Theorem \ref{thm:NrateBIP} is established for the dimension-independent convergence of the sparse quadrature error for the integration of certain quantities of interest with respect to the posterior measure \eqref{eq:PostInteGaussPara}. Before stating this theorem, let us present a summability result from \cite{cohen2010convergence} in the following lemma. 

\begin{lemma}\label{lem:ellpsum}
\cite[Theorem 7.2]{cohen2010convergence}
For any $p \leq 1$, the sequence $(\frac{|\bsmu|_1!}{\bsmu!} \bseta^\bsmu)_{\bsmu \in \cF} \in \ell^p(\cF)$ for a sequence $\bseta \in \bbR^\infty$, 
if and only if $|\bseta|_1 < 1$ and $\bseta \in \ell^p(\bbN)$.
\end{lemma}

\begin{theorem}\label{thm:NrateBIP}
Under Assumption \ref{ass:LaplaceOperator}, suppose a functional $Q: X \to \bbR$ satisfies
\beq\label{eq:boundDkQ}
\int_{\Xi} |D^{k}_m Q(m^1(\bsxi)) (\psi^1_{j_1},\cdots, \psi^1_{j_k})|^2 d\mu(\bsxi) < C k!, \quad k = 0, 1, \dots\;,
\eeq
with a constant $C < \infty $, where 
$D^{k}_m Q(m^1(\bsxi)) (\psi^1_{j_1},\cdots, \psi^1_{j_k})$ is the $k$-th derivative of $Q$ with respect to $m$ acting in the eigendirections $(\psi^1_{j_1},\cdots, \psi^1_{j_k})$ of $\cC_1$ in \eqref{eq:PostCovariance}, $j_i \in \bbN$, $i = 1, \dots, k$, and evaluated at $m^1(\bsxi)$ in \eqref{eq:postm1}. Then there exists an admissible index set $\Lambda_N$ with the $N$ indexes $\bsnu \in \Lambda_N$ corresponding to the smallest $b_\bsnu$ in \eqref{eq:bnu}, and there exists a constant $C$ independent of $N$, such that the convergence of the sparse quadrature error is bounded by
\beq\label{eq:NboundGeneral}
|\bbE^\mu(Q) - \cQ_{\Lambda_N}(Q)| \leq C (N+1)^{-s}, \quad \forall \; s < \frac{\alpha}{d}-1\;.
\eeq  
\end{theorem}

\begin{proof}
By Theorem \ref{thm:N-termConv}, to obtain the bound of the convergence rate \eqref{eq:NboundGeneral} we only need to find a sequence $(\tau_j)_{j\geq 1}$ such that the regularity assumption \eqref{eq:regularitytau} can be verified for the integrand $g(\bsxi) = Q(m^1(\bsxi))$. In fact, we have
\beq\label{eq:1postbound}
\begin{split}
& \sum_{|\bsmu|_\infty \leq r} \frac{\bstau^{2\bsmu}}{\bsmu!} \int_{\Xi} |\partial^\bsmu_\bsxi g(\bsxi)|^2 d\mu(\bsxi) \\
& = \sum_{|\bsmu|_\infty \leq r} \frac{\bstau^{2\bsmu} (\lambda^1)^{\bsmu}}{\bsmu!} \int_{\Xi} |D^{|\bsmu|_1}_m Q(m^1(\bsxi))(\psi^1)^\bsmu|^2 d\mu(\bsxi)\\
& \leq C \sum_{|\bsmu|_\infty \leq r} \frac{|\bsmu|_1!}{\bsmu!} \bseta^{\bsmu} \leq C \sum_{\bsmu \in \cF} \frac{|\bsmu|_1!}{\bsmu!}\bseta^{\bsmu}
\end{split}
\eeq 
where we have denoted $(\psi^1)^\bsmu = (\psi_{j_1}, \cdots, \psi_{j_{k}})$ with $j_{i} \in \bbJ_\bsmu$ exhausting $|\bsmu|_1$ indexes in the support of $\bsmu$, $i = 1, \dots, k $ with $k = |\bsmu|_1$, and $\bseta = (\eta_j)_{j\geq 1}$ with $\eta_j = \tau_j^2 \lambda_j^1$. In the equality, we used the chain rule for the derivative $\partial_{\xi_j} Q(m^1(\bsxi)) = D_m Q(m^1(\bsxi)) \partial_{\xi_j} m^1(\bsxi)$ and $\partial_{\xi_j} m^1(\bsxi) = (\lambda_j^1)^{1/2} \psi_j^1$. The first inequality is due to the assumption \eqref{eq:boundDkQ} and the second inequality is due to $\{|\bsmu|_\infty \leq r\} \subset \cF$, see \cite[Theorem 4.1]{bachmayr2015sparse}. By Lemma \ref{lem:ellpsum} with $p = 1$, we only need to find $\bstau$ such that $|\bseta|_1 < 1$. By Assumption \ref{ass:LaplaceOperator} and Lemma \ref{lem:uncertaintyReduction}, we set the sequence $(\tau_j)_{j\geq 1}$ as $\tau_j \sim j^\beta$ with $2\beta - 2\alpha/d < - 1$, i.e., $\beta < \alpha/d - 1/2$, such that $\sum_{j\geq 1} \tau_j^2 \lambda_j^1 < 1$. Hence, we have $(\tau_j^{-1})_{j\geq 1} \in \ell^q(\bbN)$ for any $q > 1/\beta$. Consequently, by Theorem \ref{thm:N-termConv} we obtain the convergence rate $N^{-s}$ for any $s = 1/q - 1/2 < \alpha/d - 1$.
\end{proof}

\begin{remark}
In practice, the eigenfunctions $(\psi^1_j)_{j\geq 1}$ of $\cC_1$ are not given for the verification of the assumption \eqref{eq:boundDkQ}. In such case, a stronger condition \pnote{I would keep stronger instead of using sufficient as \eqref{eq:boundDkQ} is already a sufficient condition} for $Q$ is 
\beq\label{eq:boundDkQX}
\int_{\Xi} ||D^k_m Q(m^1(\bsxi))||_{(X')^k}^2 d\mu(\bsxi) < C k!, \quad k = 0, 1, \dots\;,
\eeq
where $X'$ is the dual space of $X$. \eqref{eq:boundDkQX} implies \eqref{eq:boundDkQ} by the orthonormality of $(\psi^1_j)_{j\geq 1}$ in $X$. In the case $\max_{x\in D, j\geq 1} |\psi^1_j(x)| < C$ for some constant $C$, we can also assume
\beq\label{eq:boundDkQ1}
\int_{\Xi} |D^{k}_m Q(m^1(\bsxi)) (1,\cdots, 1)|^2 d\mu(\bsxi) < C^{2k} k!, \quad k = 0, 1, \dots\;,
\eeq
then the convergence \eqref{eq:NboundGeneral} is followed for $\tau_j \sim j^\beta$ such that $\sum_{j\geq 1} \tau_j^2 \lambda_j^1 < 1/C^2$.
\end{remark}

The result in the above theorem holds for general (nonlinear) Bayesian inverse problems as long as the integrand in \eqref{eq:PostIntePara} satisfies the assumption \eqref{eq:boundDkQ}. It allows very quick growth, $O(k!)$, of the $k$-th derivative of $Q$ with respect to the parameter $m$.
In particular, we obtain an explicit result for the integration \eqref{eq:PostInteGaussPara} in the case of linear Bayesian inverse problems as follows.

\begin{theorem}\label{thm:InverseConv}
For any bounded linear functional $f$ of the parameter $m^1$ \eqref{eq:postm1}, let $Q$ be its $k$-th power, $k = 1, 2, \dots$, i.e., $Q(m^1) = f^k(m^1)$, there exists an admissible index set $\Lambda_N$ with the $N$ indexes $\bsnu \in \Lambda_N$ corresponding to the smallest $b_\bsnu$ in \eqref{eq:bnu}, and there exists a constant $C$ independent of $N$, such that the convergence of the sparse quadrature error is bounded by
\beq\label{eq:Nbound}
|\bbE^\mu(Q) - \cQ_{\Lambda_N}(Q)| \leq C (N+1)^{-s}, \quad \forall \; s < \frac{\alpha}{d}-1\;.
\eeq
\end{theorem}
\begin{proof}
By Theorem \ref{thm:N-termConv}, to obtain the bound of the convergence rate \eqref{eq:Nbound} we only need to find a sequence $(\tau_j)_{j\geq 1}$ such that the regularity assumption \eqref{eq:regularitytau} can be verified for the quantity of interest $Q$. In fact, for $Q(m^1(\bsxi)) = f(m^1(\bsxi))$, we have
\beq\label{eq:bound1post}
\begin{split}
& \sum_{|\bsmu|_\infty \leq r} \frac{\bstau^{2\bsmu}}{\bsmu!} \int_{\Xi} |\partial^\bsmu_\bsxi Q(m^1(\bsxi))|^2 d\mu(\bsxi) \\
& = \int_{\Xi} |f(m^1(\bsxi))|^2 d\mu(\bsxi) +  \sum_{j\geq 1} \tau_j^2 \int_{\Xi} |\partial_{\xi_j}f(m^1(\bsxi))|^2 d\mu(\bsxi)\\
& = f^2(m_1) + \sum_{j\geq 1} \lambda^1_j f^2(\psi_j^1) \bbE^\mu[\xi_j^2] + \sum_{j\geq 1} \tau_j^2 \lambda_j^1 f^2(\psi_j^1) \bbE^\mu[1]\\
& \leq f^2(m_1) + ||f||^2_{X'} \sum_{j\geq 1} \lambda^1_j  + ||f||^2_{X'} \sum_{j\geq 1} \tau_j^2  \lambda_j^1\;,
\end{split}
\eeq 
where the first equality is due to the linearity of $f$ and the definition of $m^1$ in \eqref{eq:postm1}, the first inequality is due to the boundedness of $f$ and the orthonormality of $\psi_j^1$, i.e., $f^2(\psi_j) \leq ||f||^2_{X'} ||\psi_j^1||_X = ||f||^2_{X'}$. Therefore, the assumption \eqref{eq:regularitytau} is verified if $\sum_{j\geq 1} \lambda_j^1 < \infty$ and $\sum_{j\geq 1} \tau_j^2 \lambda_j^1 < \infty$. The former is ensured by Lemma \ref{lem:uncertaintyReduction}. As for the latter, by Assumption \ref{ass:LaplaceOperator} and Lemma \ref{lem:uncertaintyReduction}, we find the sequence $(\tau_j)_{j\geq 1}$ as $\tau_j \sim j^\beta$ with $2\beta - 2\alpha/d < - 1$, i.e., $\beta < \alpha/d - 1/2$. Hence, we have $(\tau_j^{-1})_{j\geq 1} \in \ell^q(\bbN)$ for any $q > 1/\beta$. Consequently, by Theorem \ref{thm:N-termConv} we obtain the convergence rate $N^{-s}$ for any $s = 1/q - 1/2 < \alpha/d - 1$.

For any $k > 1$, with $g(\bsxi) = f^k(m^1(\bsxi))$ we have 
\beq
\begin{split}
& \sum_{|\bsmu|_\infty \leq r} \frac{\bstau^{2\bsmu}}{\bsmu!} \int_{\Xi} |\partial^\bsmu_\bsxi Q(m^1(\bsxi))|^2 d\mu(\bsxi) \\
& \leq  \int_{\Xi} |f^k(m^1(\bsxi))|^2 d\mu(\bsxi) +  k \sum_{j\geq 1} \tau_j^2 \lambda_j^1 f^2(\psi_j^1) \int_{\Xi} |f^{k-1}(m^1(\bsxi))|^2 d\mu(\bsxi)\\
& + \cdots + k! \sum_{j_1, \dots, j_k \geq 1} (\tau_{j_1}^2\lambda_{j_1}^1 f^2(\psi_{j_1}^1)) \cdots  (\tau_{j_k}^2\lambda_{j_k}^1 f^2(\psi_{j_k}^1)) \int_{\Xi} |f^{0}(m^1(\bsxi))|^2 d\mu(\bsxi)\\
& \leq \bbE^\mu[f^{2k}] + k ||f||_{X'}^2 \sum_{j\geq 1} \tau_j^2 \lambda_j^1 \bbE[f^{2(k-1)}] + \cdots + k! ||f||_{X'}^{2k} \left(\sum_{j\geq 1} \tau_j^2 \lambda_j^1\right)^k \bbE^\mu[f^0],
\end{split}
\eeq 
which is bounded as long as $\bbE^\mu[f^{2i}] < \infty $, for all $i = 0, \dots, k$, and $\beta < \alpha/d - 1/2$ for $\tau_j \sim j^\beta$. $\bbE^\mu[f^{2i}] < \infty $ are verified for $i = 0$ and $i = 1$ in \eqref{eq:bound1post}. Similarly, for any $i = 2, \dots, k$, by setting $\xi_0 = 1$, $\lambda_0^1 = 1$, and $\psi_0^1 = m_1$, we have 
\beq
\begin{split}
\bbE^\mu[f^{2i}] &= \sum_{j_1, \dots, j_{2i} \geq 0} \sqrt{\lambda_{j_1}^1} f(\psi_{j_1}^1) \cdots \sqrt{\lambda_{j_{2i}}^1} f(\psi_{j_{2i}^1}) \bbE^\mu[\xi_{j_1} \cdots \xi_{j_{2i}}]\\
& = \sum_{l_1, \dots, l_i \geq 0} \lambda_{l_1}^1 f^2(\psi_{l_1}^1) \cdots \lambda_{l_i}^1 f^2(\psi_{l_i}^1) \bbE^\mu[\xi^2_{l_1} \cdots \xi^2_{l_i}]\\
& \leq ||f||_{X'}^{2i} \left(\sum_{l\geq 0} \lambda_l^1\right)^i (2i-1)!! < \infty \;,
\end{split}
\eeq
where the equality is due to $\bbE^\mu[\xi_j^n] = 0$ for any $j \geq 1$ and odd $n \in \bbN$, the first inequality is due to $\bbE^\mu[\xi^2_{l_1} \cdots \xi^2_{l_i}] \leq \bbE^\mu[\xi^{2i}] = (2i-1)!! $, $\forall l_1, \dots, l_i \geq 0$, where $\xi \sim N(0, 1)$, and the last inequality with $\sum_{l\geq 0} \lambda_l^1 < \infty$ is due to Lemma \ref{lem:uncertaintyReduction}.
\end{proof}

\begin{remark}
Note that we need $\alpha > d$ in Theorem \ref{thm:NrateBIP} and \ref{thm:InverseConv} to guarantee the convergence of the sparse quadrature and $\alpha > 2d$ (so $s = \alpha/d - 1 \geq 1$) to ensure a convergence rate of $N_p^{-s/2} \leq N_p^{-1/2}$ with respect to the number of quadrature points $N_p$, see Remark \ref{rmk:convRate}. However, in practice, we observe faster convergence than predicted by \eqref{eq:Nbound} as it is only an upper bound, see computational evidence in \cite{chen2016adaptive}, where the asymptotic convergence rate is observed to be $N_p^{-(s+1/2)}$ with respect to the number of quadrature points $N_p$, which implies a much weaker requirement of $\alpha > d$ to achieve the convergence as fast as $N_p^{-1/2}$, see the numerical tests in Sec \ref{sec:Numerical}. 
\end{remark}
%

\subsection{Adaptive construction}

We follow \cite{chen2016adaptive} to present two algorithms for the construction of the sparse quadrature, including \emph{a priori} construction guarantees the dimension-independent convergence rate in Theorem \ref{thm:N-termConv}, and a goal-oriented \emph{a posteriori} construction based on \emph{a posteriori} error indicator -- the difference quadrature $\triangle_\bsnu(g)$ in \eqref{eq:TensorDiff} that depends on each specific function $g$. Even though it can not guarantee the dimension-independent convergence rate in theory but turns out to be as accurate as the \emph{a priori} construction in practice. 
 
\subsubsection{A priori construction}
\label{sec:aprioricons}

From Theorem \ref{thm:N-termConv} we observe that the dimension-independent convergence rate of the sparse quadrature can be achieved by choosing the admissible index set $\Lambda_N$ with indexes $\bsnu \in \cF$ corresponding to the largest value of $b_{\bsnu}$. While we can compute $b_\bsnu$ for all the indexes $\bsnu \in \cF_{r,J}$ where 
\beq
 \cF_{r,J}= \{\bsnu \in \cF: |\bsnu|_{\infty} \leq r, \text{ and } \nu_j = 0 \text{ for } j > J \}\;,
\eeq
it is expensive/unfeasible if $r$ and $J$ are very large or infinite. Thanks to the monotonic increasing property of $b_\bsnu$, see \cite{chen2016adaptive}, we can adaptively construct the admissible index set $\Lambda_N$ by Algorithm \ref{alg:SparseQuad} when $(\tau_j)_{j\geq 1}$ is arranged in increasing order. We point out that this \emph{a priori} construction does not depend on the function $g$ once the Assumption \ref{ass:DeriBound} can be verified for certain $q$, $\bstau$ and $r$. However, it is not always straightforward or possible to verify this assumption especially for nonlinear function with respect to the parameter, e.g. the integrand for the nonlinear inverse problems \eqref{eq:PostIntePara}. Due to this difficulty, we adopt a goal-oriented \emph{a posteriori} construction. 

\subsubsection{Goal-oriented a posteriori construction}
\label{sec:aposteriori}
We present a goal-oriented \emph{a posteriori} construction of the sparse quadrature based on a dimension-adaptive tensor-product quadrature. It is initially developed in \cite{gerstner2003dimension} by taking advantage of the different importance of different dimensions, or different regularity of $f$ with respect to different $y_j$, $j \in \bbJ$, which we call \emph{adaptive sparse quadrature}, whose associated grid $G_{\Lambda}$ is called \emph{adaptive sparse grid}. The basic idea is based on the following adaptive process: given an admissible index set $\Lambda$, we search an index $\bsnu \in \cF$ among the forward neighbors of $\Lambda$ ($\bsnu \in \cF $ is called a forward neighbor of $\Lambda$ if $\Lambda \cup \bsnu$ is still admissible), at which $||\triangle_\bsnu||_\cS$ is maximized, and add this index to the index set $\Lambda = \Lambda \cup \{\bsnu\}$. As the number of forward neighbors depends on the dimension $J$ (in fact, the forward neighbors of $\bsnull$ are $\bse_j$ for all $j \in \bbJ$), in high or infinite dimensions, we can not search over all the forward neighbors. In such cases, it is usually reasonable to assume that the higher the dimensions, the less important they are, as determined e.g., by the fast decaying eigenvalues in Karhunen--Lo\`eve representation of the high/infinite dimensional random field. Therefore, we can explore the forward neighbors dimension by dimension in the set (see, e.g., \cite{Schillings2013, chkifa2014high})
\beq\label{eq:ForwNeib}
\cN(\Lambda) := \{\bsnu \not \in \Lambda: \bsnu - \bse_j \in \Lambda, \forall j \in \bbJ_\bsnu \text{ and } \nu_j = 0\;, \forall j > j(\Lambda)+1\},
\eeq
where $\bbJ_\bsnu = \{j: \nu_j \neq 0\}$; $j(\Lambda)$ is the smallest $j$ such that $\nu_{j+1} = 0$ for all $\bsnu \in \Lambda$. 

The adaptive sparse quadrature can be constructed following a basic greedy algorithm proposed in \cite{gerstner2003dimension}, which was improved on the data structure in \cite{klimke2006uncertainty} to copy with very high dimensions (e.g., up to $10^4$ dimensions in a personal laptop with $16$GB memory). We present the goal-oriented \emph{a posteriori} construction in Algorithm \ref{alg:SparseQuad}. Note that for the \emph{a priori} construction in this algorithm, we do not need to evaluate step 5 and 18 if the maximum number of indexes is imposed as the only stopping criterion. We can also replace the maximum number of indexes by the maximum number of points $|G_{\Lambda \cup \cN(\Lambda)}|$. Moreover, for step 13 it is also a common practice to chose $\bsnu$ as $\bsnu = \argmax_{\bsmu \in \cN(\Lambda_N)} ||\triangle_{\bsmu}(f)||_\cS /|G_\bsmu|$ to balance the error and the work, e.g., \cite{gerstner2003dimension, nobile2016adaptive}.

\begin{algorithm}
\caption{Adaptive sparse quadrature}
\label{alg:SparseQuad}
\begin{algorithmic}[1]
\STATE{\textbf{Input: } tolerance $\epsilon$, maximum number of indexes $N_{\text{max}}$, function $f$.}
\STATE{\textbf{Output: } the admissible index set $\Lambda_N$, quadrature $\cQ_{\Lambda_N}(f)$.}
\STATE{Set $N= 1$, $\Lambda_N = \{\bsnull\}$, and compute $\cQ_{\Lambda_N}(f)$.}
\STATE{Construct the forward neighbor set $\cN(\Lambda_N)$ by \eqref{eq:ForwNeib}.}
\STATE{Compute $\triangle_\bsnu(f)$ for all $\bsnu \in \cN(\Lambda_N)$ by \eqref{eq:TensorDiff}.}
\IF {\emph{a priori} construction}
\STATE{Compute $b_{\bsnu}$ for all $\bsnu \in \cN(\Lambda_N)$ by \eqref{eq:bnu}.}
\ENDIF
\WHILE{$\max_{\bsmu \in \cN(\Lambda_N)}||\triangle_{\bsmu}(f)||_\cS > \epsilon$ \textbf{and} $N < N_{\text{max}}$}
\IF {\emph{a priori} construction}
\STATE{Take $\bsnu = \argmax_{\bsmu \in \cN(\Lambda_N)} b_\bsnu$.}
\ELSE 
\STATE{Take $\bsnu = \argmax_{\bsmu \in \cN(\Lambda_N)} ||\triangle_{\bsmu}(f)||_\cS$.}
\ENDIF
\STATE{Enrich the index set $\Lambda_{N+1} = \Lambda_N \cup \{\bsnu\}$}.
\STATE{Set $\cQ_{\Lambda_{N+1}}(f) = \cQ_{\Lambda_N}(f) + \triangle_{\bsnu}(f)$.}
\STATE{Construct the forward neighbor set $\cN(\Lambda_{N+1})$ by \eqref{eq:ForwNeib}}.
\STATE{Compute $\triangle_\bsnu(f)$ for all $\bsnu \in \cN(\Lambda_{N+1})$ by \eqref{eq:TensorDiff}.}
\STATE{Set $ N \leftarrow N + 1$}.
\ENDWHILE
\end{algorithmic}
\end{algorithm}

\section{Numerical Experiments}
\label{sec:Numerical}
In this section, we present both a linear and a nonlinear inverse problem to demonstrate the concentration of the posterior distribution and the convergence property of the sparse quadrature.

\subsection{A linear inverse problem \pnote{with analytic solution}}
\subsubsection{Problem setup}
We first consider an analytic linear inverse problem to
explicitly demonstrate the variance reduction of the uncertain parameter from its prior distribution to its posterior distribution and to illustrate the dependence of the reduction on the prior distribution and the observation noise.
In particular, we consider Poisson's equation in a physical domain $D\subset \bbR^d \; (d = 1, 2, 3)$ with homogeneous Dirichlet boundary condition: find $u \in H_0^1(D)$ such that 
\begin{equation}\label{eq:Poisson}
\begin{split}
-\triangle u &= m \quad \text{ in } D, 
\end{split}
\end{equation}
where $m$ is a Gaussian random field with prior distribution $\cN(m_0, \cC_0)$, and $\triangle$ is the Laplace operator. We consider $\cC_0 =(-\beta \triangle)^{-\alpha}$ with the smoothness parameter $\alpha > d/2$, and the variance parameter $\beta > 0$. By $\{(\lambda_j, \psi_j)\}_{j\geq 1}$ we denote the eigenpairs of $-\triangle$, so that the eigenpairs of $\cC_0$ are given by $\{((\beta\lambda_j)^{-\alpha}, \psi_j)\}_{j\geq 1}$, and we can write
\beq\label{eq:LaplaceEigen}
\cC_0 = \sum_{j \geq 1} \lambda_j^0 \psi_j \otimes \psi_j, \quad \text{ where } \lambda_j^0 = (\beta\lambda_j)^{-\alpha}.
\eeq
We assume that the observation data is given by 
\beq
y = u + \eta,
\eeq
where $\eta \sim \cN(0, \sigma^2 I)$, with $\sigma > 0$ and $I$ being the identity operator in $L^2(D)$. Conditioned on the observation data, we seek the posterior distribution of $m$, which is an infinite-dimensional linear Bayesian inverse problem. 
Due to linearity, the posterior distribution is also Gaussian, i.e., $m \sim \cN(m_{1}, \cC_{1})$ with
\beq
m_1 = \argmin_{m} \cJ(m) = \cC_1 \left((-\beta \triangle)^\alpha m_0- \sigma^{-2} \triangle^{-1}y\right).
\eeq
and
\beq
\cC_1 =  (D^2_m J(m))^{-1} = (\sigma^{-2} \triangle^{-2} + (-\beta \triangle)^{\alpha})^{-1},
\eeq
where the cost functional $\cJ(m)$ reads
\beq
\cJ(m) = \frac{1}{2 \sigma^2}||y - \triangle^{-1}m||_{L^2(D)}^2 + \frac{1}{2}||m-m_0||^2_{\cC_0}.
\eeq
Note that $\cC_1$ admits the eigendecomposition with the same eigenfunctions of $\cC_0$
\beq\label{eq:C1}
\cC_1 = \sum_{j\geq 1} \lambda_j^1 \psi_j \otimes \psi_j, \quad \text{ where } \lambda_j^1 = \frac{(\beta\lambda_j)^{-\alpha}}{\sigma^{-2} \beta^{-\alpha}\lambda_j^{-\alpha - 2} + 1}.
\eeq
Fig. \ref{fig:EigenDecayExam1} displays the decay of the square root of the eigenvalues of the prior and the posterior covariance operators with different noise level $\sigma = 10^{-1}, 10^{-2}$, different prior variance parameter $\beta = 10^{-1}, 10^{-2}$, and different smoothness parameter $\alpha = 1, 2$.

\begin{figure}[!htb]
\begin{center}
\includegraphics[scale=0.34]{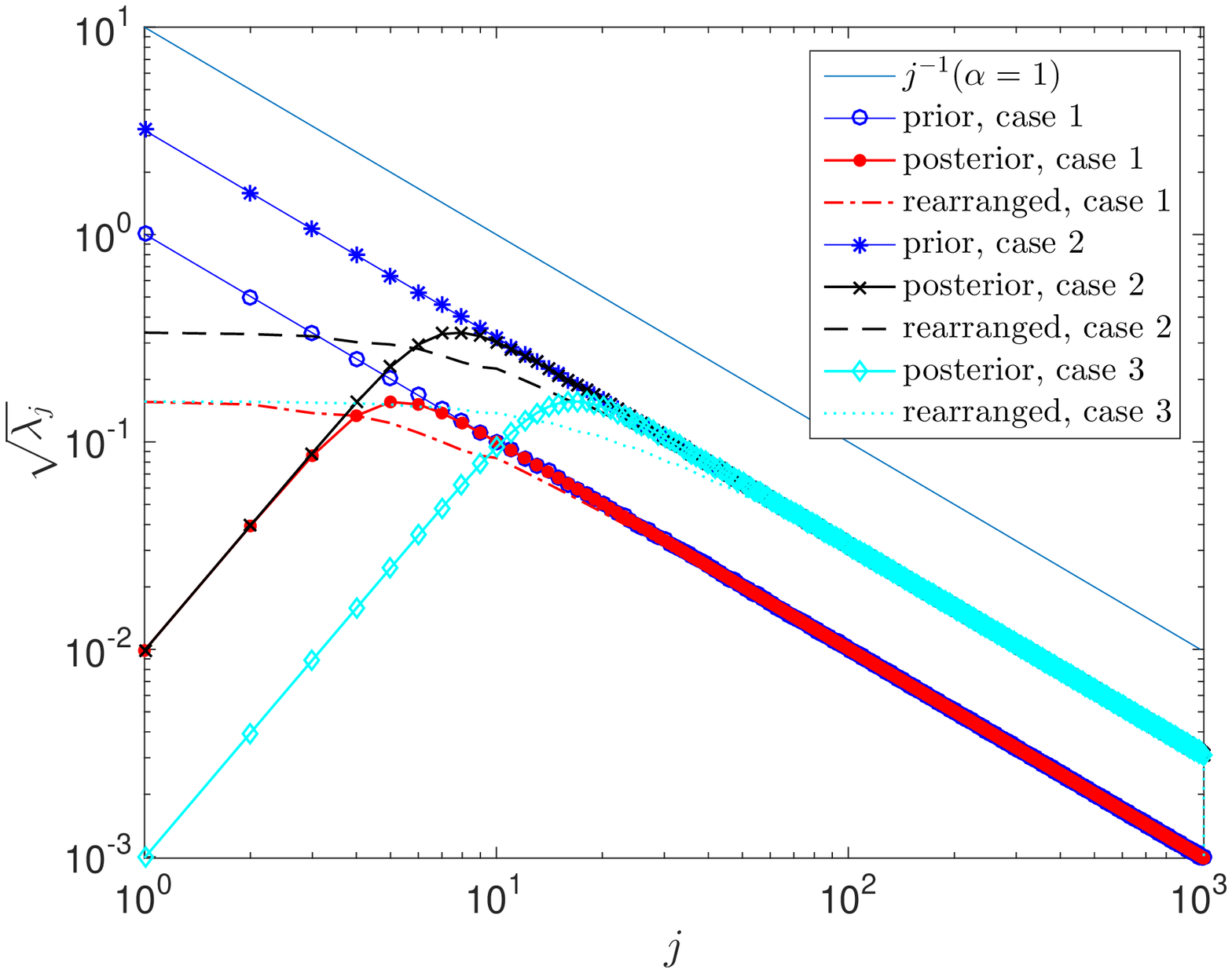}
\hspace*{0.1cm}
\includegraphics[scale=0.34]{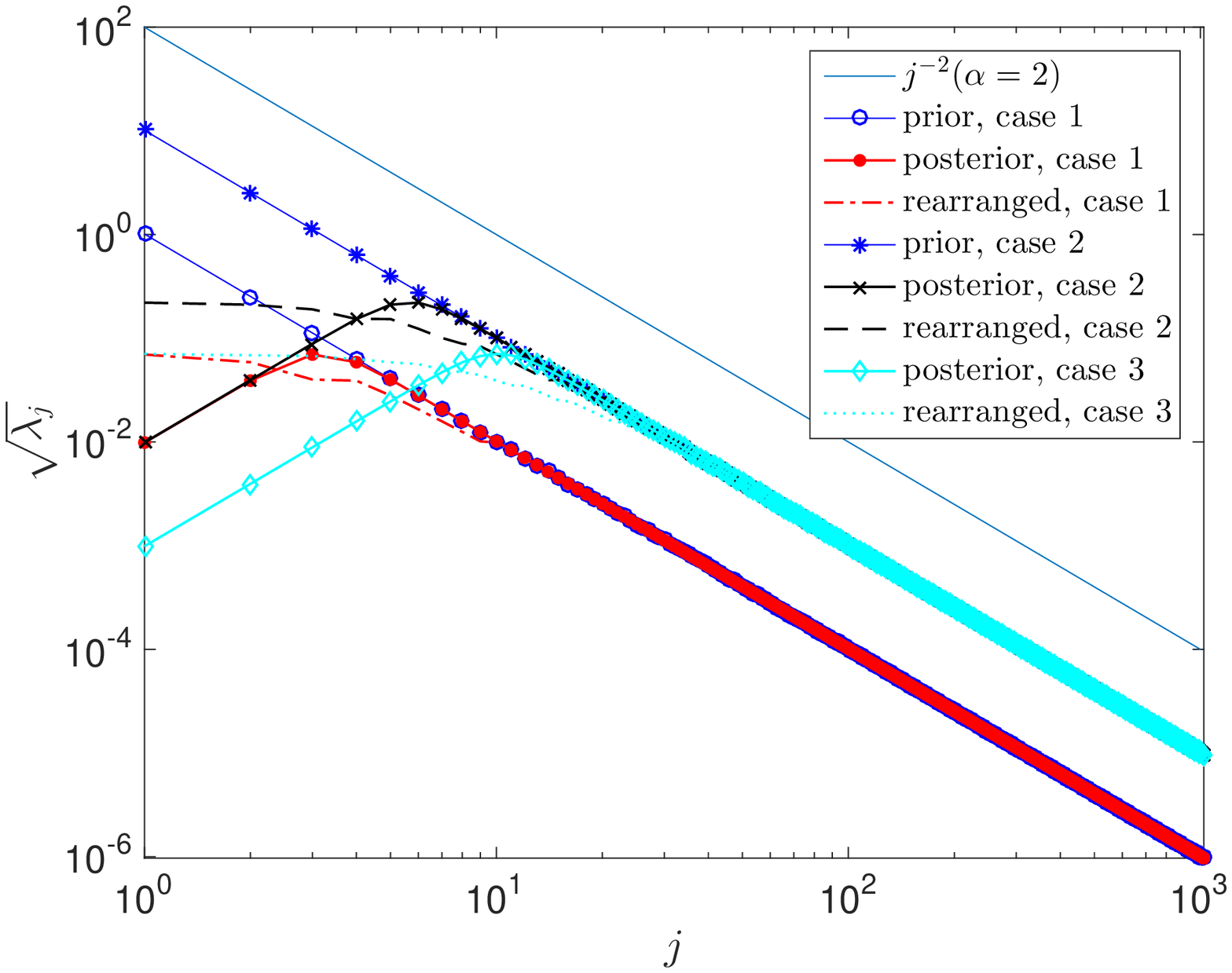}
\end{center}
\caption{Decay of the square root of the eigenvalues of the prior and the posterior covariance operators with different noise $\sigma$, variance parameter $\beta$, and smoothness parameter $\alpha$. Dashed line represents that of the posterior operator rearranged in descending order. Case 1, $(\beta, \sigma) = (10^{-1}, 10^{-2})$; case 2, $(\beta, \sigma) = (10^{-2}, 10^{-1})$; case 3, $(\beta, \sigma) = (10^{-2}, 10^{-2})$. $\alpha = 1$ (left), $\alpha = 2$ (right).}\label{fig:EigenDecayExam1}
\end{figure}

We can see that $\lambda_j^{1} < \lambda_j^{0}$ for all $j$, the uncertainty is reduced from the prior to the posterior in every dimension (eigendirection), and the reduction is more evident in the low dimensions, and less evident in high dimensions as $\lim_{j \to \infty} \lambda_j^{1} / \lambda_j^{0} = 1.$
Moreover, we can see that the smaller $\sigma$ is (smaller noise), the larger the reduction becomes, which can be interpreted as smaller noise provides more informative observation data. Furthermore, the smaller $\beta$ is, i.e., the larger the prior variance, the larger the reduction becomes. Finally, as the smoothness parameter $\alpha$ becomes larger, $\lambda_j^{-\alpha - 2}$ decays faster, so that the reduction effect is only evident in the first fewer dimensions.
Note that the eigenvalues of the posterior covariance operator do not decay monotonically. To use the sparse quadrature that is adaptively constructed from low dimension to high dimension, we rearrange these eigenvalues in decreasing order.

\subsubsection{Numerical test}
We set the physical domain $D = (0,1)$ and use piecewise linear finite element in a uniform mesh with mesh width $h = 2^{-10}$ for the discretization in physical space, which leads to 1023 parameter dimensions (note that $m(0) = m(1) = 0$). The eigenpairs of the Laplace operator with Dirichlet boundary conditions in $D$ read $\lambda_j = \pi^2 j^2$ and $\psi_j = \sin(j\pi x)$, $j \geq 1$. We write the Karhunen--Lo\`eve expansion of the parameter $m$ as follows
\beq\label{eq:1023KL}
m^{(i)} = m_i + \sum_{j =1}^{1023} \sqrt{\lambda_j^{i}} \xi_j \psi_j, \quad \xi_j \stackrel{i.i.d.}{\sim} \cN(0,1), \quad i = 0, 1, 
\eeq
where $i = 0$ and $i = 1 $ represent the expansion w.r.t. the prior and the posterior, respectively. 
To generate the observational data $y$, we first take a random sample $m_s$ from the prior distribution $\cN(m_0, \cC_0)$ (with $m_0 = 0$, $\beta = 5\times 10^{-2}$, and $\alpha = 1$) and we compute the solution $u_s$ of the forward problem \eqref{eq:Poisson} using $m_s$ as forcing term. Then we set $y = u_s + \eta$, where $\eta \sim \cN(0, \sigma^2 I)$ and  $\sigma = 10^{-2}$. 

In Fig. \ref{fig:anchoredDist} we show one and two-dimensional \emph{anchored marginals} from the prior and the posterior distribution (recall that due to linearity the posterior is also Gaussian). The anchored marginals are computed with respect to the first (i.e., dominant) 6 dimensions of the prior distribution. More explicitly, the one dimensional plots are obtained (up to a rescaling factor) by setting the $\xi_j$ in the expansion \eqref{eq:1023KL} to $0$  for all indexes $j$ except index $i$ ($i=1,\ldots,6$) which is allowed to vary in $\bbR$; the two dimensional plots are obtained by setting all the $\xi_j$ in the expansion \eqref{eq:1023KL} to $0$ expect the two indexes $k,l \in \{1, \dots, 6\}$.
Note how the posterior density is concentrated in a small region in the support of the prior density, especially for the first few dimensions. This concentration of the posterior may cause prior-based sparse quadrature to yield an inaccurate integration w.r.t. the posterior distribution; as shown in the middle-top part of the figure, Gauss--Hermite quadrature points w.r.t. the prior distribution may completely fail in detecting the posterior distribution. In contrast, quadrature points computed using the Hessian-based parametrization well capture the posterior distribution, as shown in the top-right part of the figure.

\begin{figure}[!htb]
\begin{center}
\includegraphics[scale=0.8]{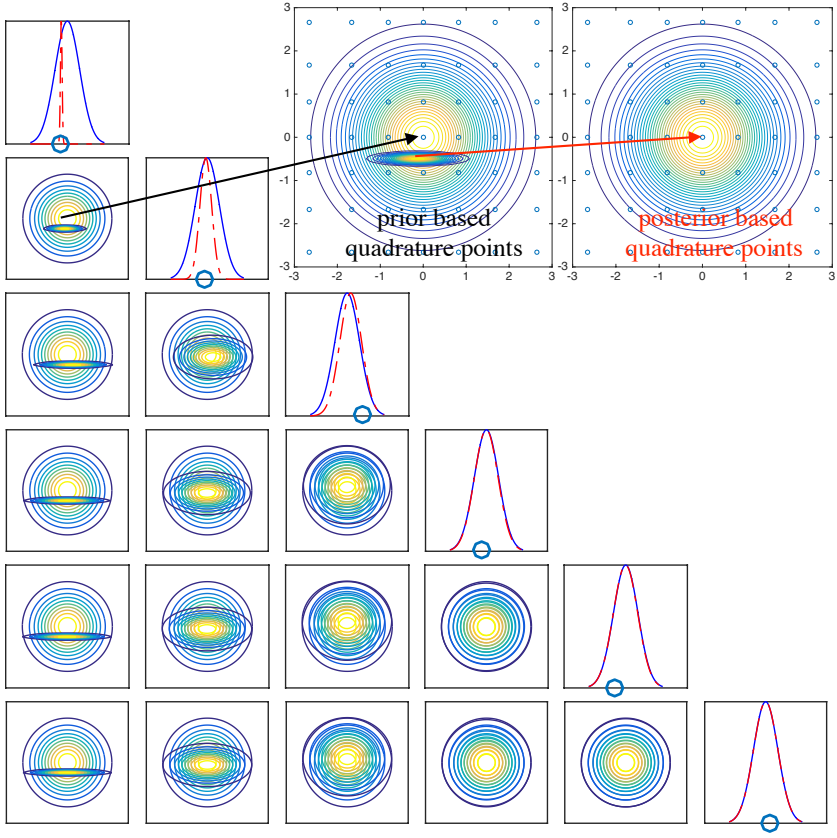}
\end{center}
\caption{Anchored marginal prior (blue solid line) and posterior (red dashed line) densities with the sample parameter value (in circle) in the first 6 dimensions as well as the contour of their mutual joint densities. The top-middle part is enlarged from the joint (both prior and posterior) density of dimension 1 and 2 with additional prior-based tensor-product Gauss--Hermite quadrature points (in circle). The top-right part is the corresponding transformed joint posterior density with Hessian-based quadrature points.
The posterior density is divided/normalized by its maximum value.}\label{fig:anchoredDist}
\end{figure}

To demonstrate the convergence properties of the Hessian-based adaptive sparse quadrature, we consider two quantities of interest. The first is
$$
Q_1(m) = e^{m(0.5)},
$$ 
which is a nonlinear function of parameter $m$ evaluated at $x = 0.5$, with all the derivatives bounded by 
$$
\bbE^\mu[|D_m^kQ_1(m^1)(\psi_{j_1}, \dots, \psi_{j_k})|^2] \leq \bbE^\mu[e^{2m(0.5)}] < \exp\big(2m_1(0.5) + 2|\lambda^1|_1\big)\;,
$$ 
which satisfies the assumption in Theorem \ref{thm:NrateBIP}. The second is
$$
Q_2(m) = (10\times u'(0.5))^2.
$$
which is the square of the first derivative of the solution $u$ at $x = 0.5$.
Note that $Q_2(m)$ \pnote{depends quadratically on $m$ (as the derivative of the solution of the forward problem $u'$ is bounded and linear with respect to $m$)}, align with Theorem \ref{thm:InverseConv}. Note that the delta functional $\delta_{0.5}(m) = m(0.5)$ acting the eigenfunctions $\sin(\pi j x)$ is uniformly bounded by $1$. We use the approximation $u'(0.5) \approx (u(0.5+h)-u(0.5-h))/(2h)$ for mesh size $h$. Both integrals $\bbE^{\mu^y}[Q_1]$ and $\bbE^{\mu^y}[Q_2]$ can be explicitly computed. More explicitly,
\begin{equation}\label{eq:Q1discrete}
\bbE^{\mu^y}[Q_1]  = \exp\left( \bsm_1^\top \bse_{0.5} + \sum_{j\geq 1} \lambda_j^1 (\bspsi_j^\top \bse_{0.5})^2 \right), 
\end{equation} 
where $\bsm_1$ is the coefficient vector of the MAP point $m_1$ in the finite element space, $\bse_{0.5} = (0, \dots, 0, 1, 0, \dots, 0)^\top$ is a vector of the same size as $\bsm_1$ with one at $x = 0.5$ and zero everywhere else, $\bspsi_j$ is the coefficient vector of the eigenfunction. For $Q_2$,
\begin{equation}\label{eq:Q2discrete}
\bbE^{\mu^y}[Q_2] = \bsd^\top_{0.5} \bbA^{-1} \bbM (\bsm_1 \bsm_1^\top + \bbC_1) \bbM \bbA^{-1} \bsd_{0.5},
\end{equation}
where $\bbC_1$ is the covariance matrix of $\bsm^1$, $\bsd_{0.5} = 10\times (0, \dots, 0, -1, 0, 1, 0, \dots, 0)^\top/(2h)$ corresponding to the derivative vector at $x = 0.5$, $\bbA$ and $\bbM$ are the stiffness and mass matrix corresponding to the Laplace operator $-\triangle$ and the identity operator $I$. 
%

\begin{figure}[!htb]
\begin{center}
\includegraphics[scale=0.33]{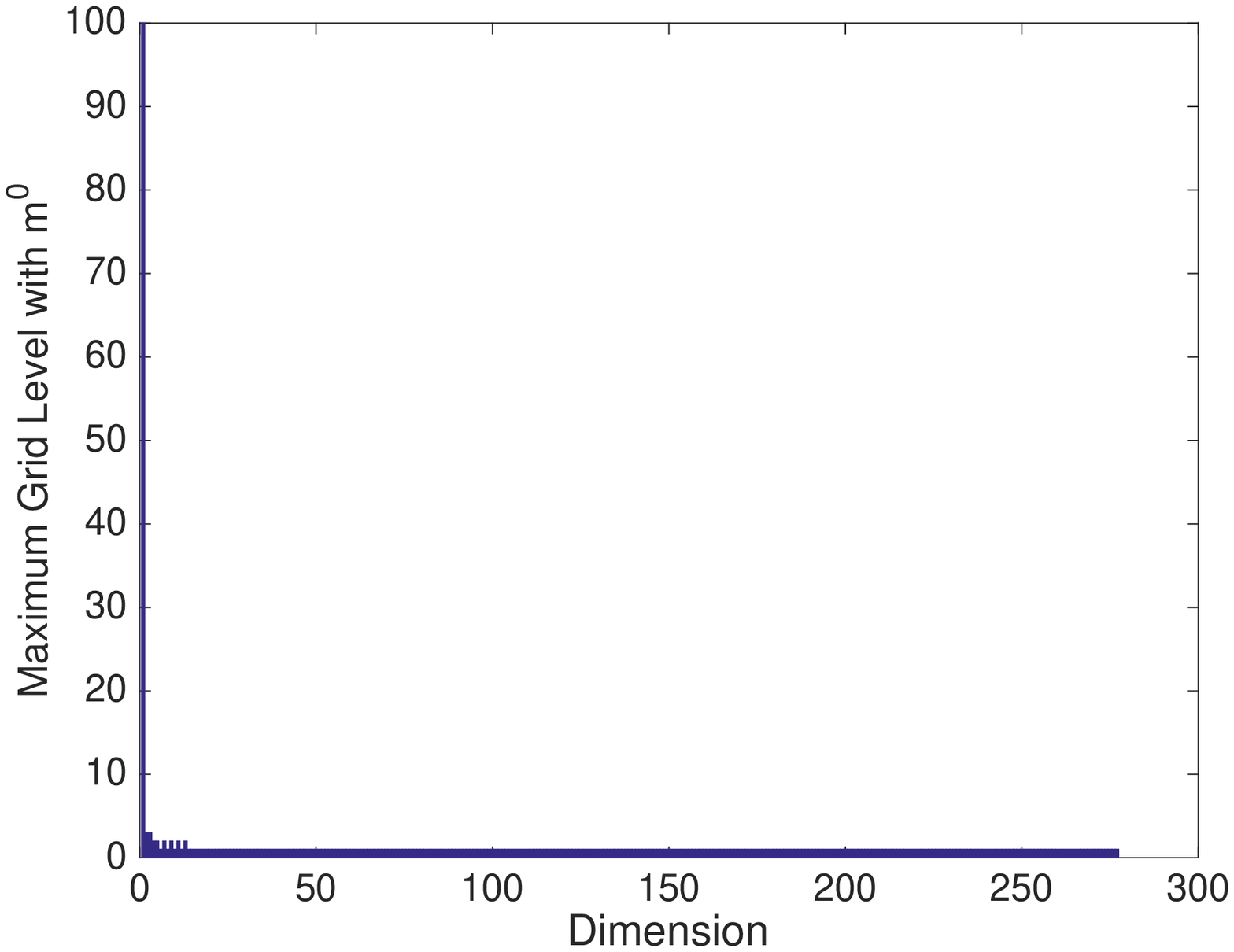}
\hspace*{0.2cm}
\includegraphics[scale=0.33]{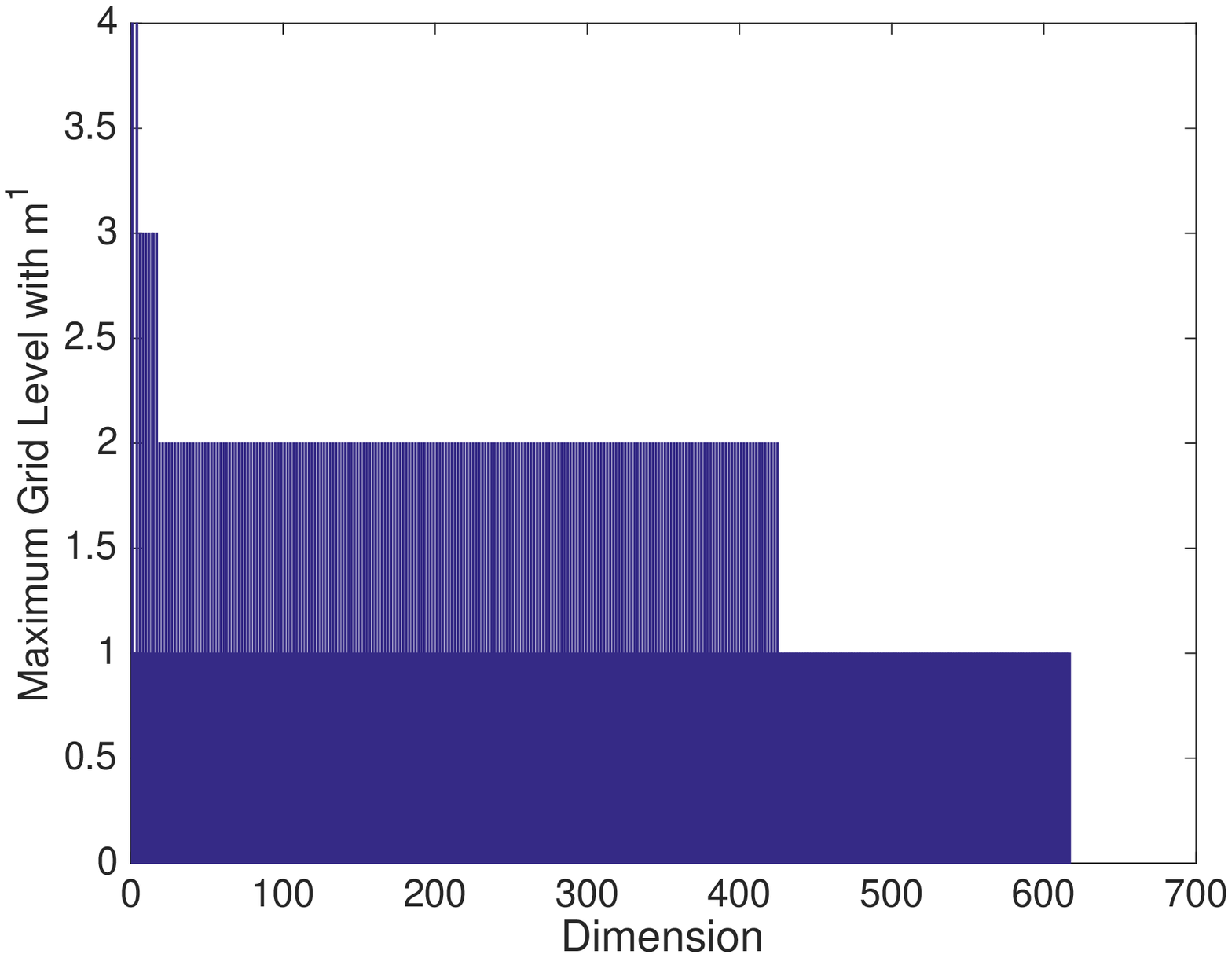}
\end{center}
\caption{The maximal level of the sparse quadrature constructed in each dimension by the prior-based (left) and the Hessian-based (right) quadratures for the integration of $Q_1$ with $\alpha = 1$.}\label{fig:sparselevellinear}
\end{figure}

We run the adaptive sparse quadrature by Algorithm \ref{alg:SparseQuad} for both the prior-based (with parametrization $m^0 \sim \cN(m_0, \cC_0)$) and the Hessian-based (with parametrization $m^1 \sim \cN(m_1, \cC_1)$) integration formulae \eqref{eq:PostIntegralPara} and \eqref{eq:PostInteGaussPara}, with both \emph{a priori} and \emph{a posteriori} construction methods. 
In this example, the concentration of the posterior in a small region of the support of the prior distribution (see Fig. \ref{fig:anchoredDist}) caused the prior-based adaptive sparse quadrature algorithm to terminate prematurely without adequate sampling of the posterior distribution and providing an erroneous integration result. To compute the quadrature errors for the two QoI we use the values computed in \eqref{eq:Q1discrete} and \eqref{eq:Q2discrete} as their references. 
Fig. \ref{fig:sparselevellinear} displays the maximum sparse grid level in each dimension constructed by the prior-based and Hessian-based sparse quadrature for the integration of $Q_1$ with $\alpha = 1$. We can observe that the prior-based sparse quadrature allocates most of the quadrature points in the first dimension due to the concentration of the posterior distribution in the domain of the prior distribution along the first dimension, while the Hessian-based sparse quadrature does not show such concentration. By increasing the number of quadrature points, more and more dimensions are activated. We can see that with $10^5$ points the first $617$ dimensions are activated by the Hessian-based sparse quadrature construction. We remark that \pnote{$Q_1$ does not depend on $j$ if $j$ is even since $\psi_j^1(0.5) = \sin( 0.5 \pi j) = 0$, so only the first grid level (with one quadrature point)} is activated in the even dimensions, see the right of Fig. \ref{fig:sparselevellinear}.

\begin{figure}[!htb]
\begin{center}
\includegraphics[scale=0.35]{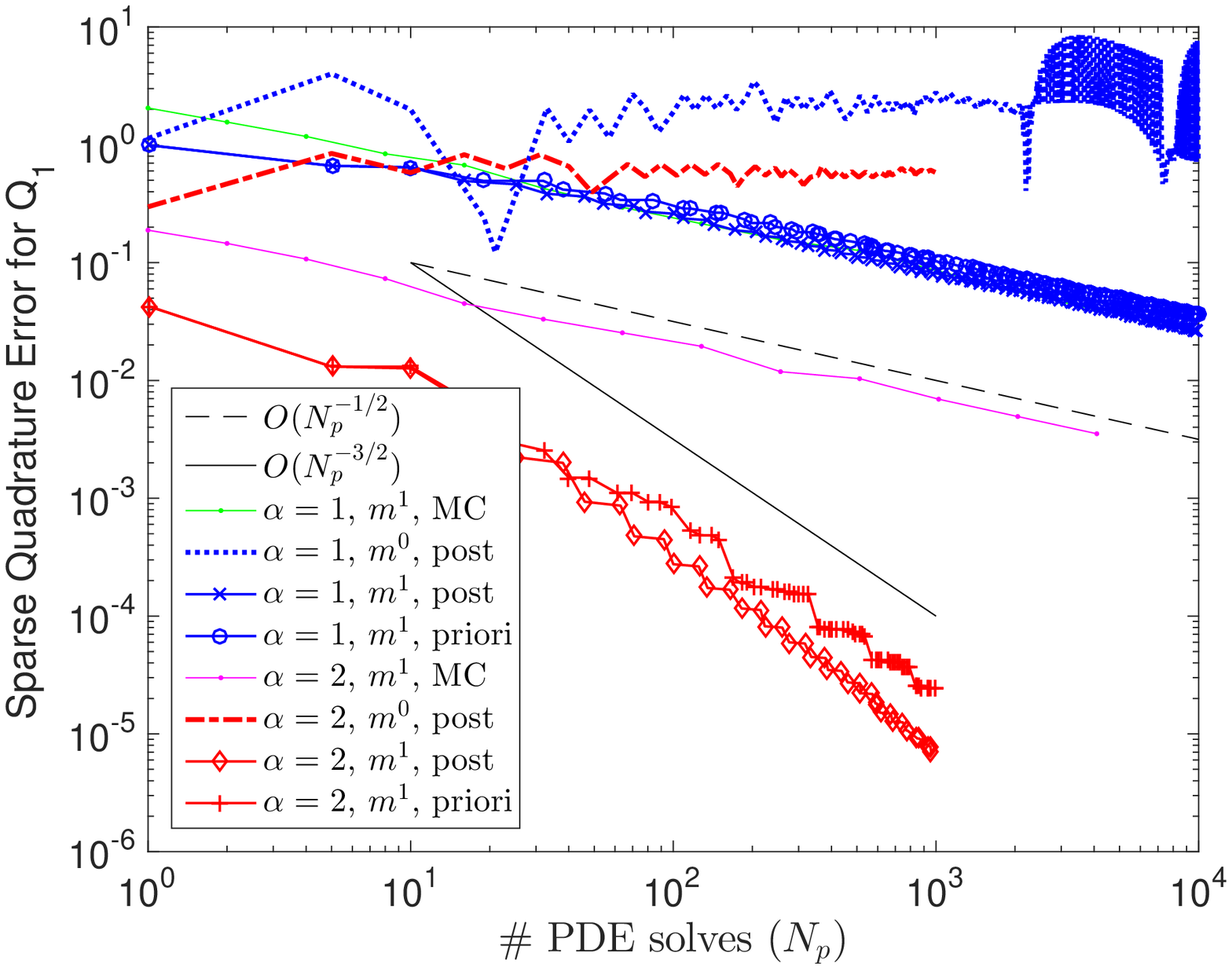}\hspace*{0.1cm}
\includegraphics[scale=0.35]{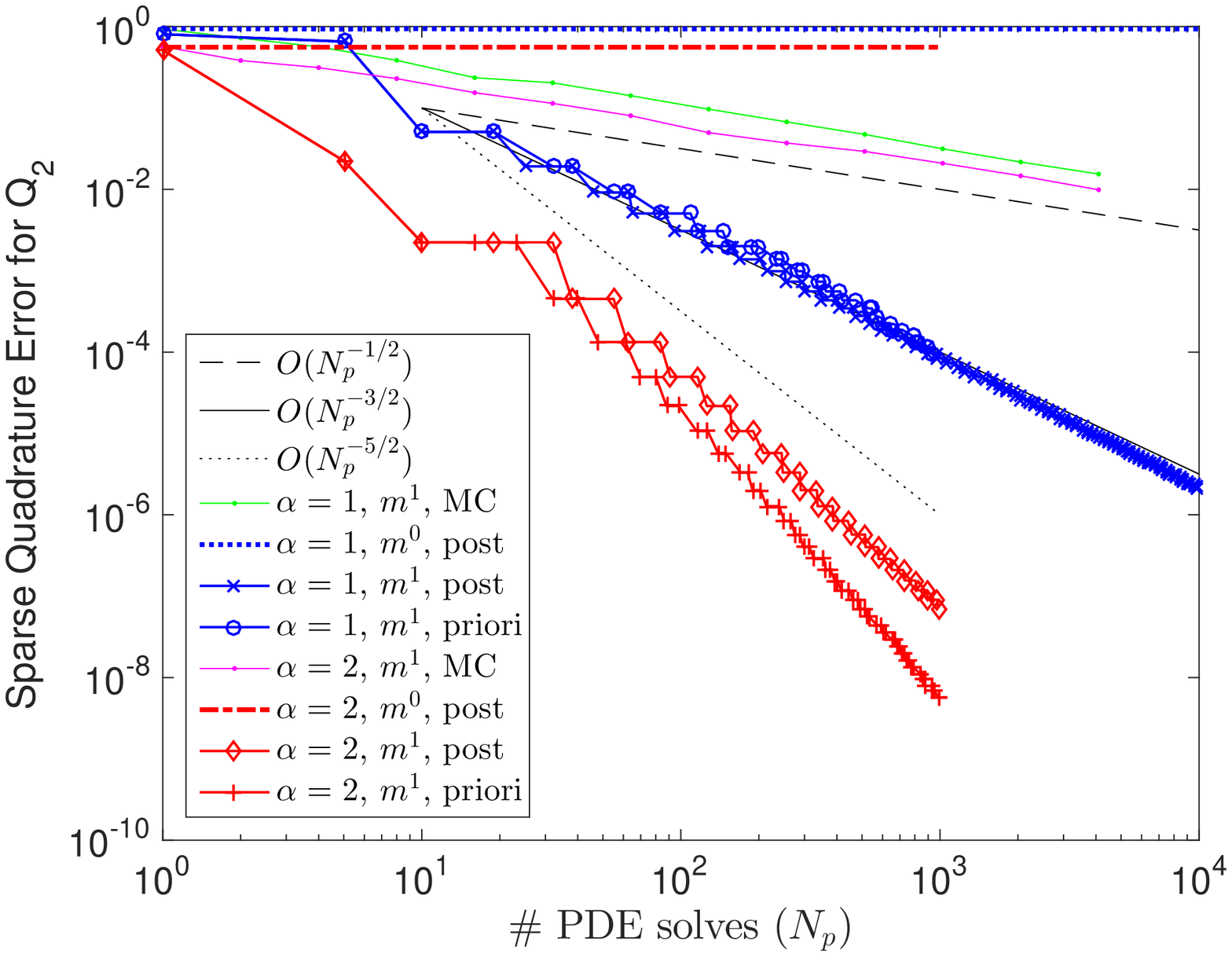}
\end{center}
\caption{Decay of the adaptive sparse quadrature errors w.r.t. the number of PDE solves by using the prior-based ($m^0$) and the Hessian-based ($m^1$) quadratures, with both \emph{a priori} and \emph{a posteriori} (post) construction methods, for the computation of the integrals of $Q_1$ (left) and $Q_2$ (right). The decay of the Hessian-based Monte Carlo (MC) quadrature errors (the average of 100 trials) are also shown for both integrals with both $\alpha = 1$ and $\alpha = 2$. 
}\label{fig:ConvRate}
\end{figure}

In Fig. \ref{fig:ConvRate} we plot the convergence of the sparse quadrature errors as a function of the number of quadrature points (or PDE solves). We can see that
the prior-based quadrature errors do not converge and remain large for both the two quantities of interest. In contrast, the Hessian-based quadrature errors by both the \emph{a priori} and the \emph{a posteriori} constructions converge with rate independent of the (active) parameter dimension. For $Q_1$, the \emph{a posteriori} construction leads to smaller quadrature errors than the \emph{a priori} construction, which is observed in most other cases, see \cite{chen2016adaptive}. However, the opposite is observed for $Q_2$ as the \emph{a priori} construction activates more dimensions than the \emph{a posteriori} construction for $Q_2$. Note that $Q_2$ is only quadratic with respect to $\xi_j$, $j \geq 1$, thus one level of Gauss--Hermite quadrature is exact for each dimension, so that more activated dimensions result in more accurate quadrature. 
The asymptotic convergence rate for $Q_1$ is $N_p^{-s}$ with $s = 1/2$ for $\alpha = 1$ and $s = 3/2$ for $\alpha = 2$, which is $s = \alpha/d  - 1/2$, larger than the bound $\alpha/d - 1$ in Theorem \ref{thm:NrateBIP} and \ref{thm:InverseConv}. For $Q_2$ we observe $N_p^{-s}$ with $s = 3/2$ for $\alpha = 1$
and $s = 5/2$ for $\alpha = 2$. The faster convergence for $Q_2$ is because of the higher regularity of the derivative of the solution $u'$ than that of the parameter $m$ itself in the sense of Assumption \ref{ass:DeriBound}. This also indicates that the convergence property of the sparse quadrature does not only depend on the parametrization but also on the regularity of the quantity of interest. We also plot the decay of the averaged Monte Carlo (MC) quadrature errors of 100 trials for the Hessian-based integration, from which we can observe the convergence rate $N_p^{-1/2}$ in all cases, even for the integrand with high regularity, e.g., $\alpha = 2$. 

\subsection{A nonlinear inverse problem}
\subsubsection{Problem setup}
We consider a nonlinear inverse problem where the forward model is an elliptic equation describing Darcy flow in a porous medium $D\subset \bbR^d \; (d = 1, 2, 3)$: find $u \in V := H^1_{\Gamma_D}(D)$ such that 
\beq\label{eq:Darcyflow}
-div(e^m \nabla u) = f, \quad \text{ in } D, 
\eeq  
with suitable Dirichlet boundary condition on the boundary $\Gamma_D$ and zero Neumann boundary condition on the rest of the boundary.
We assume that the diffusion coefficient is a lognormal random field, i.e., the parameter $m$ obeys a Gaussian distribution with mean $m_0$ and covariance $\cC_0$, i.e., $m \sim \cN(m_0, \cC_0)$. 
Moreover, we assume that measurements of $m$ are available at a few locations $x^l \in D$, $l = 1, \dots, L$. The mean of the parameter $m_0$ is obtained as a solution of the optimization problem 
\beq\label{eq:meanprior}
m_0 = \argmin_{m} \left\{ \frac{1}{2}\langle m, \cA m \rangle + \frac{\kappa}{2} \langle m-m_{true}, \cM (m-m_{\rm true})\rangle, \right\}
\eeq 
where $\langle \cdot, \cdot \rangle$ denotes the inner product in $L^2(D)$, $\cA = - \beta \triangle + \gamma I$, $m_{\rm true}$ is the true parameter field, $\kappa$ is a penalization parameter, and $\cM$ is a measurement operator defined as 
\beq
\cM = \sum_{l=1}^L \varepsilon_l I, \quad \text{ where }  \varepsilon_l  = \exp\left(-\frac{(x-x^l)^2}{2r^2}\right),
\eeq
with suitable radius $r > 0$. The covariance of $m$ is defined as 
\beq
\cC_0 = (\cA + \kappa \cM)^{-\alpha}, \quad \text{ with } \alpha > d/2.
\eeq 
We specify the parameter-to-observable map $\cG$ in \eqref{eq:obsdata} as 
\beq
\cG(m) = \cB\, u(m) \in \bbR^K,
\eeq
where $\cB = (b_1, \dots, b_K)^T$ is a vector of linear observation functionals $b_k \in V^*$, $k=1,\ldots, K$, where
\beq
\langle b_k, u \rangle = \int_D \exp\left(-\frac{(x-x^k_{obs})^2}{2r_{obs}^2}\right) u(x) dx.
\eeq 
Therefore, $\cG$ is nonlinear w.r.t. $m$ as $u$ is so, leading to a nonlinear inverse problem.
We specify the covariance matrix $\Gamma_{\text{noise}} = \sigma^2 I$ for the noise $\eta$ in \eqref{eq:obsdata}, where $I$ is the identity matrix of size $K \times K$.
To compute the MAP point, we apply the inexact Newton--conjugate gradient method \cite{bui2013computational}, i.e., using Newton iteration method for the nonlinear optimization w.r.t. the parameter $m$ and conjugate gradient method for the solution of the incremental parameter $\hat{m}$ at each Newton step, for which we need to compute the gradient and Hessian of the cost functional $\cJ$ (defined in \eqref{eq:cost}) constrained by problem \eqref{eq:Darcyflow}.
We introduce 
the Lagrangian functional 
\beq
\cL(u,m,p) = \cJ(m) + \langle e^m\nabla u, \nabla p\rangle - \langle f, p \rangle.
\eeq
where $p$ is the Lagrangian multiplier. Then the gradient of $\cJ$ can be computed as 
\beq
D_m \cJ(m) (\tilde{m})= \partial_m \cL(u,m,p) (\tilde{m}) = \langle m-m_0, \tilde{m} \rangle_{\cC_0} + \langle \tilde{m} e^m \nabla u, \nabla p \rangle, \; \forall \tilde{m} \in X,
\eeq
where $u$ and $p$ are the solutions of the state and adjoint equations obtained by the first order variation of the Lagrangian functional 
\beq
\begin{split}
\langle e^m \nabla u, \nabla \tilde{p} \rangle &= \langle f,\tilde{p}\rangle,  \quad \forall \tilde{p} \in V,\\
\langle e^m \nabla \tilde{u}, \nabla p \rangle & = \langle \cB^* \Gamma_{\text{noise}}^{-1} (y - \cB u), \tilde{u}\rangle, \quad \forall \tilde{u} \in V;
\end{split}
\eeq
and the Hessian of $\cJ$ evaluated at $\hat{m} \in X$ can be computed as 
\beq
\begin{split}
& D^2_m \cJ(m)(\tilde{m}, \hat{m})  = \partial^2_m \cL(u,m,p) (\tilde{m}, \hat{m}) \\
& = \langle \tilde{m} e^m \nabla \hat{u}, \nabla p \rangle + \langle \tilde{m} e^m \nabla u, \nabla \hat{p} \rangle +  \langle \hat{m}, \tilde{m} \rangle_{\cC_0} + \langle \tilde{m} \hat{m} e^m \nabla u, \nabla p \rangle, \; \forall \tilde{m} \in X,
\end{split}
\eeq
where $\hat{u}$ and $\hat{p}$ are the solutions of the incremental state and adjoint equations obtained by the second order variation of the Lagrangian functional
\beq
\begin{split}
\langle e^m \nabla \hat{u}, \nabla \tilde{p} \rangle &= - \langle \hat{m} e^m \nabla u, \nabla \tilde{p} \rangle, \quad \forall \tilde{ p} \in V,\\
\langle e^m \nabla \tilde{u}, \nabla \hat{p} \rangle &= - \langle \hat{m} e^m \nabla \tilde{u}, \nabla p \rangle - \langle \cB^* \Gamma_{\text{noise}}^{-1} \cB \hat{u}, \tilde{u} \rangle, \quad \forall \tilde{u} \in V.
\end{split}
\eeq

\subsubsection{Numerical test} We consider the physical domain $D = (0,1)$ with Dirichlet boundary condition $u(0) = 1$ and $u(1) = 0$, and specify the parameters for the prior distribution as 
$\alpha = 1, \beta = 2, \gamma = 1, \kappa = 10^3$, and the measurement locations $x^l = (l-1) 2^{-2}$ with $l = 1, \dots, 5$.  For the observation noise we set $\sigma = 5\times 10^{-2}$ at the observation locations $x_{obs}^k = (k-1) 2^{-6}$ with $k = 1, \dots 65$.  We use piecewise linear finite element for the discretization at a uniform mesh of width $h = 2^{-10}$, thus resulting in a $1025$-dimensional problem. We set the radius for the measurement and the observation operators $r = r_{obs} = h$. 

\begin{figure}[!htb]
\begin{center}
\includegraphics[scale=0.5]{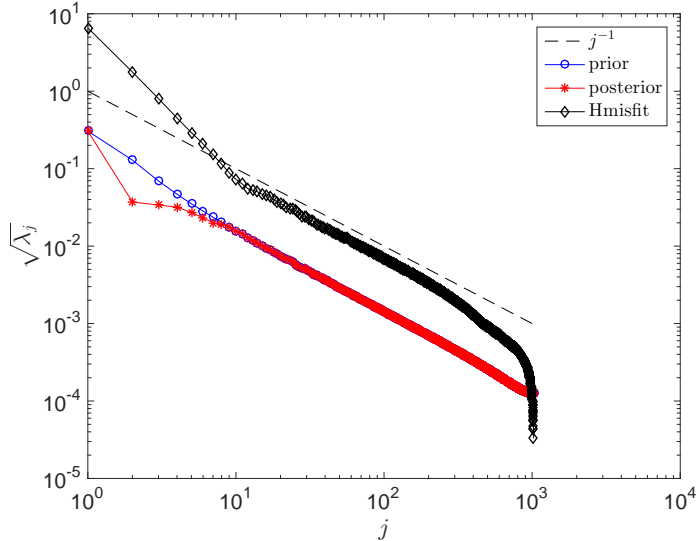}
\end{center}
\caption{
Decay of the square root of the eigenvalues of the prior covariance in \eqref{eq:priorEigenvalue}, the Gauss-posterior covariance in \eqref{eq:postEigenvalue}, and the prior-preconditioned Hessian misfit in \eqref{eq:HmisfitEigenvalue}. 
}\label{fig:sampleeigen}
\end{figure}

To generate the prior mean $m_0$ and the observation vector $y$, we first sample the \emph{true} parameter $m_{\rm true}$ from the Gaussian distribution $\cN(0, \cA^{-1})$. Then we let $m_0$ be the solution of the quadratic optimization problem \eqref{eq:meanprior}, and we set $y = \cB\,u_{\rm true} + \eta$, where $u_{\rm true}$ is the solution of the forward problem \eqref{eq:Darcyflow} with $m=m_{\rm true}$ and $\eta \sim \cN({\bf 0}, \Gamma_{\rm noise})$ is a random noise vector.

In Fig. \ref{fig:sampleeigen}, we plot the square root of the eigenvalues $(\sqrt{\lambda_j})_{j\geq 1}$ of the prior covariance $\cC_0$ in \eqref{eq:priorEigenvalue}, the Gaussian posterior covariance $\cC_1$ in \eqref{eq:postEigenvalue}, as well as the prior-preconditioned Hessian misfit term $ H_{\text{misfit}}$ in \eqref{eq:HmisfitEigenvalue} for completeness. The former two converge with an asymptotic rate of $O(j^{-1})$ for the sparsity parameter $\alpha = 1$. The latter decay very fast for the first few dimensions which are the most informed by the observation data.
We can observe a reduction of the variance from the prior to the Gaussian approximate of the posterior in these dimensions. Note that, for this particular problem with $f=0$ in \eqref{eq:Darcyflow}, the data do not inform the constant component of $m$; for this reason we do not observe reduction of the variance for the first dimension of the prior distribution whose corresponding eigenvector is nearly constant in space.


\begin{figure}[!htb]
\begin{center}
\includegraphics[scale=0.75]{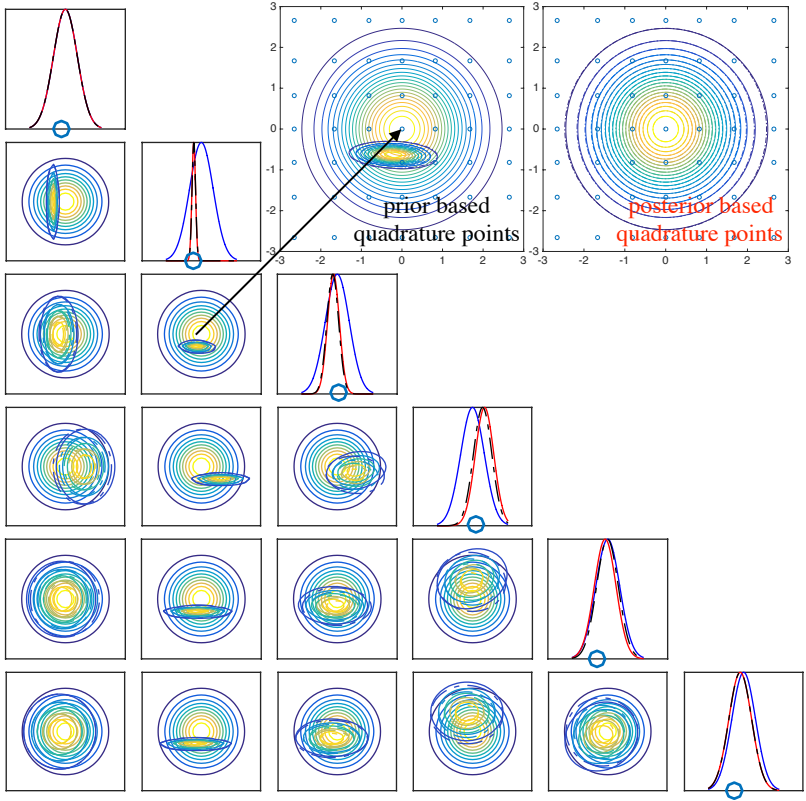}
\end{center}
\caption{Anchored marginal prior (blue solid line), posterior (red solid line), and the Gaussian approximate of the posterior (black dashed line) densities with the sample parameter value (in circle) in the first 6 dimensions as well as the contour of their mutual joint densities (dashed line for Gaussian posterior). The top-middle part is enlarged from the joint densities of dimension 2 and 3 with additional prior-based tensor-product Gauss--Hermite quadrature points (in circle). The top-right part is the joint posterior and the joint Gaussian approximate of the posterior densities (note that they are very close to each other) of dimension 2 and 3 with Hessian-based quadrature points.}\label{fig:anchoredDistPrior}
\end{figure}

Fig. \ref{fig:anchoredDistPrior} displays the anchored marginal densities of the prior $\rho^0(m) = \exp(-\frac{1}{2}||m - m_0||^2_{\cC_0})$, the posterior $\rho^{post} = \exp(-\frac{1}{2}||y - \cG(m)||_{\Gamma_{noise}}^2)\rho^0(m)$, and the Gaussian approximate of the posterior $\rho^1 = \exp(-\frac{1}{2}||m - m_1||^2_{\cC_1})$,
as well as their mutual joint density in the first 6 dimension of the KL expansion of the parameter w.r.t. its prior distribution. Note that all the densities are divided (normalized) by their maximal value. The data are particularly informative in the second dimension, leading to a much smaller variance of the posterior distribution with respect to the prior distribution, as shown by how  the posterior density concentrates in a small region. The prior-based quadrature points fail to capture the local shape of the posterior density, thus leading to inefficient sparse quadrature. On the contrary, the posterior and the Gaussian approximate of the posterior densities are close to each other in each of the 6 leading dimensions in the parametrization of the prior distribution. Moreover, as shown in the top-right part of Fig. \ref{fig:anchoredDistPrior}, the Hessian-based quadrature points effectively capture the shape of the posterior distribution.


We run Algorithm \ref{alg:SparseQuad} to compute the integration of the QoI 
$$
Q = u(0.5)
$$ 
using both the prior-based quadrature \eqref{eq:PriorIntegral} and the Hessian-based quadrature \eqref{eq:PostIntegraldu0du1}. 
We need to compute the integrals with the two integrands $Q_1$ and $Q_2$ given by
$$
Q_1 = \exp(-\Phi(m)) \text{ and } Q_2 = Q \exp(-\Phi(m))
$$ for the prior-based quadrature \eqref{eq:PriorIntegral} and by
$$Q_1 = \exp(-\cJ_1(m)) \text{ and } Q_2 = Q \exp(-\cJ_1(m))$$ for the Hessian-based quadrature \eqref{eq:PostIntegraldu0du1}. We remark that due to the complexity of the posterior density and the nonlinear dependence of the solution $u$ on the parameter $m$, we can not rigorously verify the assumption for the two integrands in Theorem \ref{thm:NrateBIP}. The decay of the adaptive sparse quadrature errors is shown in Fig. \ref{fig:convnonlinear}, where the reference values to compute the quadrature errors are taken as the adaptive sparse grid quadrature results with $10^4$ points for the prior-based quadrature and $10^5$ points for the Hessian-based quadrature, respectively.
The error of prior-based quadrature rule does not decay and remains oscillating for both $Q_1$ and $Q_2$; on the contrary, the error of the Hessian-based quadrature decays with an asymptotic rate of $O(N_p^{-1})$ for $Q_1$ and $Q_2$. We remark that in Fig. \ref{fig:convnonlinear} we only show the quadrature errors of the sparse quadrature obtained by the \emph{a posteriori} construction in Algorithm \ref{alg:SparseQuad}, since the \emph{a priori} construction based only on the coefficient $b_\bsnu$ in \eqref{eq:bnu} failed to effectively detect the complicated shape of two QoIs in the parameter space and resulted in too much oscillation of the quadrature errors. This numerical test suggests that \emph{a posteriori} construction should be used for nonlinear inverse problems as the posterior distribution/density function could have rather complicated shape. Moreover, we point out that the Gaussian approximate of the posterior distribution is quite close to the true posterior distribution as shown in Fig. \ref{fig:anchoredDistPrior}. If the Gaussian approximate is very poor, which occurs for informative observation data with very small noise or high non-linearity of the forward model, we caution that the Hessian-based allocation of the quadrature points may not capture the posterior distribution so that the sparse quadrature could fail to approximate the required integrals. How to tackle this difficulty is an ongoing research.

\begin{figure}[!htb]
\begin{center}
\includegraphics[scale=0.33]{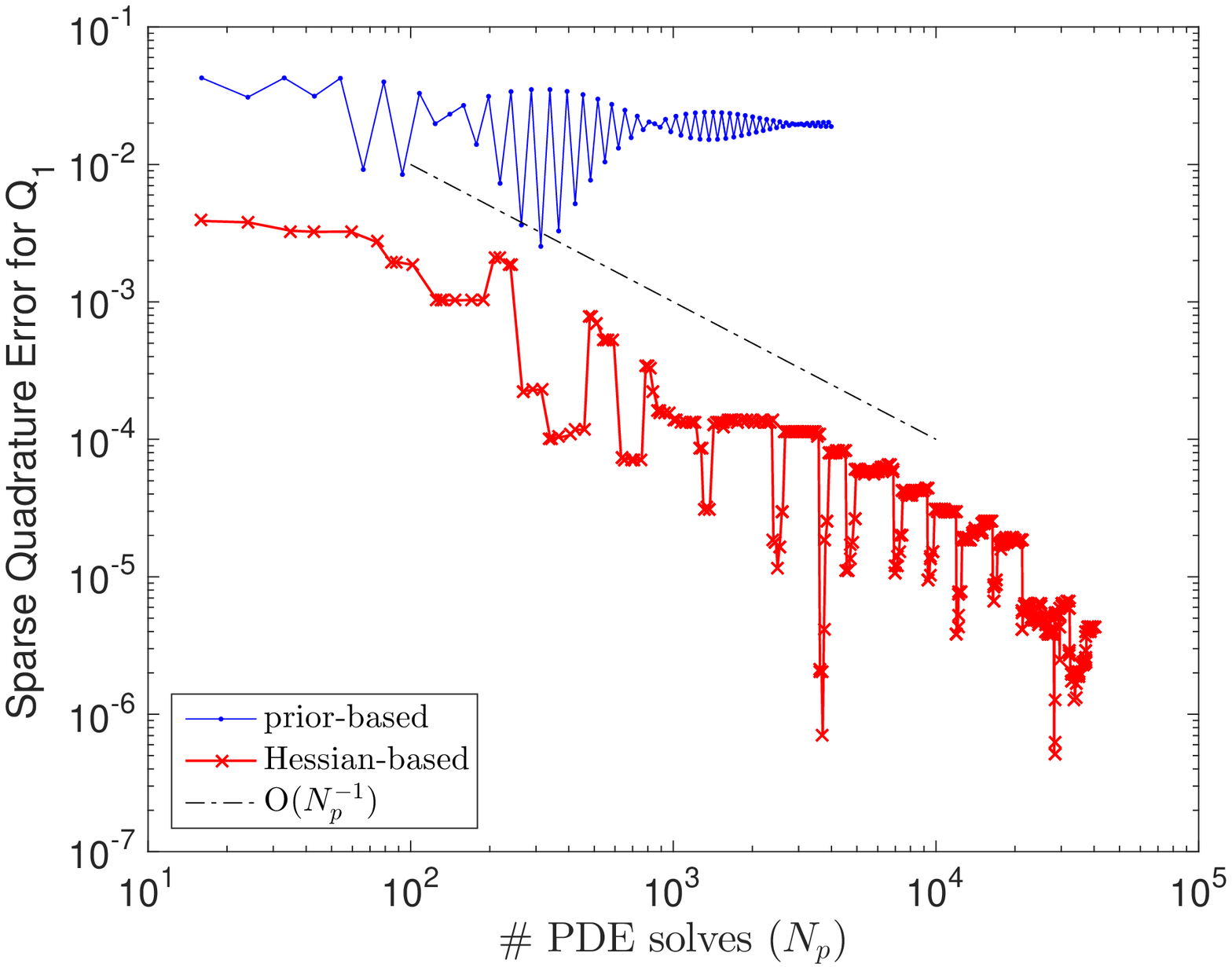}
\hspace*{0.2cm}
\includegraphics[scale=0.33]{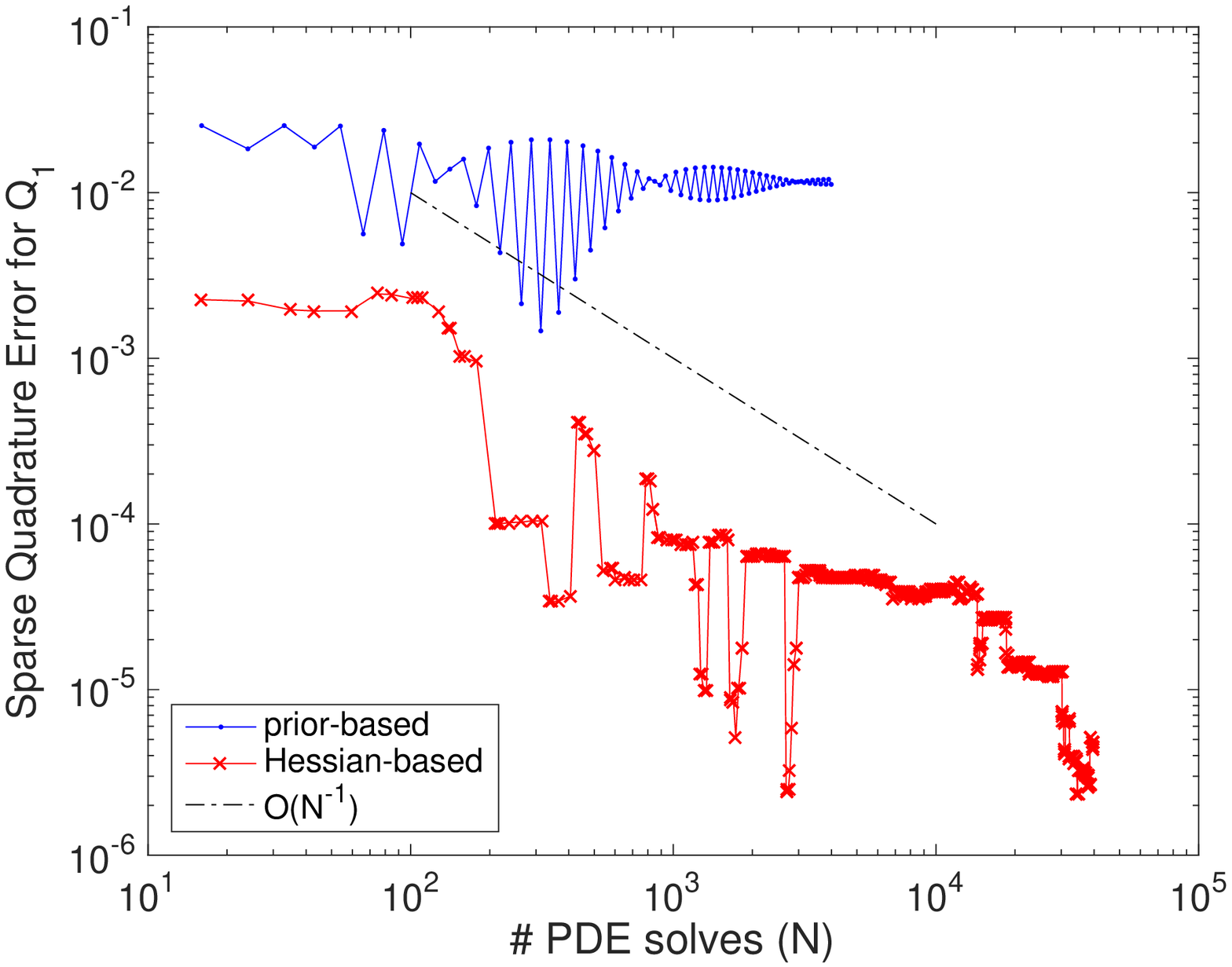}
\end{center}
\caption{Decay of the sparse quadrature errors for the $Q_1$ (left) and $Q_2$ (right), using both the prior-based quadrature \eqref{eq:PriorIntegral} and the Hessian-based quadrature \eqref{eq:PostIntegraldu0du1} with \emph{a posteriori} construction. 
}
\label{fig:convnonlinear}
\end{figure}

\section{Conclusions}
\label{sec:conclusion}
In this work, we proposed and analyzed a Hessian-based adaptive sparse
quadrature to compute infinite-dimensional integrals with respect
to the posterior distribution in the context of
infinite-dimensional Bayesian inverse problems. Dimension-independent
convergence  of the proposed method was theoretically
established for linear inverse problems and demonstrated numerically
for both linear and nonlinear inverse problems.  
The convergence of the adaptive sparse quadrature is faster than that
of the Monte Carlo or Markov-chain Monte Carlo method, especially for
sparser problems, i.e., for problems in which  the eigenvalues of the
covariance operator decay faster and the
parameter-to-quantity-of-interest map is smoother. We illustrated the
concentration of the posterior distribution in a small region of the
support of the prior distribution for both linear and nonlinear
inverse problems. We showed that a prior-based parametrization
quadrature may fail to capture the shape of the posterior distribution
and result in inaccurate values of integrals with respect to the posterior
distribution, while the Hessian-based parametrization effectively
allocates quadrature points in regions of high density of the
posterior distribution and leads to correct values of the integrals.
Further investigation and application of our proposed method is
ongoing for Bayesian inverse problems with very small observation
noise and highly nonlinear forward problems. These lead to more complicated
and localized posterior distributions that cannot be effectively
captured by the Hessian-based transformation, but may be with a
measure transport \cite{marzouk2016introduction}. 

\vspace*{0.5cm}

\appendix
\section{A re-parametric formulation}
\label{subsec:ParaFormulation}
In this section, we briefly review the re-parametric formulation presented in \cite{schillings2014scaling}, where $m$ is first pre-parametrized according to the prior measure $\mu_0 = \cN(m_0,\cC_0)$ and then re-parametrized using Hessian-based directions.
We then compare our development in Section \ref{subsec:HessianPara} with such re-parametric formulation both in terms of computational efficiency and accuracy for the solution of Bayesian inverse problems constrained by large-scale PDE models. To start, we denote the eigenpairs of the covariance operator $\cC_0$ as $(\lambda_j^0,\psi_j^0)_{j\geq 1}$ and denote 
\beq
\Psi_0 = (\psi^0_1, \psi^0_2, \dots, ), \quad \Lambda_0 = \text{diag}(\lambda_1^0, \lambda_2^0, \dots, )\;,
\eeq
so that the parameter $m$ can be expressed by the Karhunen--Lo\`eve expansion as 
\beq\label{eq:compactm}
m(\bsxi) = m_0 + \sum_{j \geq 1} \sqrt{\lambda_j^0} \psi_j^0 \xi_j \equiv m_0 + \Psi_0 \Lambda_0^{1/2} \bsxi\;.
\eeq
The prior measure of $m$ can be fully represented by that of $\bsxi$, which is Gaussian $\cN(0,I)$, where $0 = (0,0, \dots,)$ and $I = \text{diag}(1,1, \dots, )$. The MAP point of the posterior measure is a parameter $\bsxi_1 \in \Xi$ that solves the following minimization problem 
\beq\label{eq:projectedMinimizationMAP}
\min_{\bsxi \in \Xi} \cJ(m(\bsxi))\;,
\eeq
where the $\bsxi$-parametric cost functional is given explicitly by \eqref{eq:cost} and \eqref{eq:compactm}
\beq
\cJ(m(\bsxi)) = \frac{1}{2}||y - \cG(m(\bsxi))||_{\Gamma_{\text{noise}}}^2 + \frac{1}{2} \bsxi^\top \bsxi\;.
\eeq
Note that we have used the eigendecomposition $\cC_0^{-1} = \Psi_0 \Lambda_0^{-1} \Psi_0^\top$ in the second term.
We remark that the solution $\bsxi_1 \in \Xi$ implies that $m_1 \in E$, the Cameron--Martin space associated with the prior measure $\cN(m_0, \cC_0)$. The Hessian of the $\bsxi$-parametric cost functional $\cJ$ evaluated at the MAP point is given by  
\beq
H_{\text{MAP}}^\bsxi = D_\bsxi^2 \cJ(m(\bsxi))|_{\bsxi = \bsxi_1} = H_{\text{misfit}}^\bsxi  +  I\;,
\eeq
where Hessian of the misfit term is obtained by applying the chain rule for derivative
\beq
H_{\text{misfit}}^\bsxi  = \left(\frac{Dm}{D\xi}\right)^\top \left(D_{m}^2 \Phi(m)\right) \frac{Dm}{D\xi}\Big|_{\bsxi = \bsxi_1} =  \Lambda_0^{1/2} \Psi_0^\top H_{\text{misfit}} \Psi_0 \Lambda_0^{1/2}\;.
\eeq
Consequently, the covariance of the Gaussian measure $\cN({\bsxi_1}, \cC_1^\bsxi)$  is given by 
\beq
\cC_1^\bsxi = \left(H_{\text{MAP}}^\bsxi\right)^{-1} = \left(
H_{\text{misfit}}^{\bsxi} + I
\right)^{-1} = \Lambda_0^{-1/2} \Psi_0^\top \cC_1 \Psi_0 \Lambda_0^{-1/2}\;.
\eeq
Since the covariance operator $\cC_1$ can be decomposed as $\cC_1 = \Psi_1 \Lambda_1 \Psi_1^\top$, we have 
\beq
\cC_1^\bsxi  =  \left(\Lambda_0^{-1/2} \Psi_0^\top \Psi_1 \Lambda_1^{1/2}\right) \left(\Lambda_1^{1/2} \Psi_1^\top \Psi_0 \Lambda_0^{-1/2}\right) =: L L^\top\;,
\eeq
where by the compact and self-adjoint property of $\cC_1^\bsxi$, i.e. there exist the eigenpairs of $\cC_1^\bsxi$, denoted as $(\Lambda_1^\bsxi, \Psi_1^\bsxi)$, and there holds $\cC_1^\bsxi = \Psi_1^\bsxi \Lambda_1^\bsxi (\Psi_1^\bsxi)^\top$, so that we can identify $L = \Psi_1^\bsxi (\Lambda_1^\bsxi)^{1/2}$, so that $L^{-1} = (\Lambda_1^{\bsxi})^{-1/2} (\Psi_1^\bsxi)^\top$.
We make the change of variables 
by shifting ${\bsxi_1}$, rotating $(\Psi_1^\bsxi)^\top$ and rescaling $(\Lambda_1^\bsxi)^{-1/2}$ for $\bsxi$, i.e.
\beq
\bszeta = L^{-1} (\bsxi - {\bsxi_1})\;, 
\eeq
for which we have $\bbE^\bsrho[\bszeta] = 0$ and $\bbE^\bsrho[\bszeta \bszeta^\top]$ = I, so that $\bszeta \in \cN(0,I)$; moreover, we have 
\beq
 \bsxi = L \bszeta + {\bsxi_1}\;\;.
\eeq
Replacing this $\bsxi$ as a function of $\bszeta$ in \eqref{eq:compactm}, we obtain 
\beq\label{eq:repam}
m(\bsxi(\bszeta)) =  m_0 +  \Psi_0 \Lambda_0^{1/2} {\bsxi_1}+ \Psi_0 \Lambda_0^{1/2} L \bszeta = m_1 + \Psi_1 \Lambda_1^{1/2} \bszeta\;,
\eeq
where $m_1 = m(\bsxi_1)$. Denoting $m^1(\bszeta)= m(\bsxi(\bszeta))$, we obtain the same parametrization of the parameter $m$ as in \eqref{eq:postm1} derived in Section \ref{subsec:HessianPara}. 

Our Hessian-based parametrization and the above re-parametric formulation are therefore equivalent in the idealized case of no truncation in the Karhunen--Lo\`eve expansion of $m$ with respect to the prior. However, in practice, only a finite number of eigenmodes in such expansion can be computed.
In \cite{schillings2014scaling}, the authors first perform a finite-dimensional truncation of the pre-parametric parameter \eqref{eq:compactm}, then solve the minimization problem for the MAP point \eqref{eq:projectedMinimizationMAP} in the subspace spanned by the dominant eigenmodes of the prior, and finally compute the Hessian and its eigenmodes projected onto such subspace.
This may lead to discretization bias with respect to the stochastic parameter, since directions that are well informed by the data may not be correctly captured by the subspace of the dominant eigenmodes of the prior. In contrast, our formulation does not involve truncation of the parameter according to the prior measure and, therefore, allows to capture all the directions --- resolved by the spatial discretization of the parameter field --- that are informed by the data.

\section{Computation of the parametrization}
\label{subsec:comHesspara}
Let $X_h$ denote a finite-dimensional subspace of $X$ with $N_h$ basis functions $\{\phi_h^n\}_{n=1}^{N_h}$, for instance finite element space or spectral space, where $h$ represents the mesh size or the reciprocal of the spectral order, then we can approximate the parameter $m$ by 
\beq
m_h = \sum_{n=1}^{N_h} m_n \phi_h^n,
\eeq
where $\bsm = (m_1, \dots, m_{N_h})^T$ is the vector of the coefficients in the basis function expansion. Let $\bbM \in \bbR^{N_h\times N_h}$ denote the mass matrix with $\bbM_{ij} = \langle \phi_h^j, \phi_h^i\rangle, \;, i,j = 1, \dots, N_h$. Let $\bbA \in \bbR^{N_h\times N_h}$ denote the matrix corresponding to the elliptic differential operator \eqref{eq:elldiffoper}, given by 
\beq
\bbA_{ij} = \langle \phi_h^i, \cA \phi_h^j \rangle = \beta \langle \Theta \nabla\phi_h^j, \nabla \phi_h^i \rangle + \gamma \langle \phi_h^j, \phi_h^i \rangle, \quad i,j = 1, \dots, N_h.
\eeq
Then we can compute the following inner products as 
\beq
\langle m_1, \cA^\alpha m_2 \rangle = \bsm_1^\top \bbA_\alpha \bsm_2, \text{ and } \langle m_1, \cA^{-\alpha} m_2 \rangle = \bsm_1^\top \bbM\bbA_\alpha^{-1} \bbM \bsm_2
\eeq
where $\bbA_\alpha = (\bbA\bbM^{-1})^{\alpha-1}\bbA$, $\alpha = 1, 2, \dots $. The discrete eigenpairs of the prior covariance operator $\cC_0$ are obtained by solving the generalized eigenvalue problem 
\beq\label{eq:priorEigenvalue}
\bbM\bbA_\alpha^{-1} \bbM \bspsi_j^0 = \lambda_j^0 \bbM \bspsi_j^0, \quad j \geq 1,
\eeq
for which we can employ a Lanczos method \cite{calvetti1994implicitly} or a randomized method \cite{halko2011finding} that involves $O(\alpha J)$ elliptic PDE solves for computing $J$ eigenpairs. Both methods get rid of direct assembling of the matrix $\bbM \bbA_\alpha^{-1}\bbM$, which could be very expensive for large-scale problems, i.e., $N_h$ large. 

To compute the discrete eigenpairs of the posterior covariance operator $\cC_1$ in \eqref{eq:PostCovariance}, we first need to compute the MAP point as the solution of the optimization problem \eqref{eq:MAP}, for which we apply a Newton-conjugate gradient method~\cite{bui2013computational, hippylib2016}, where for the solution of the increment at each Newton step, we use conjugate gradient method. At the MAP point, evaluation of the Hessian misfit term $H_{\text{misfit}}$ in a given direction can be realized by solving one forward PDE and three additional linear PDEs corresponding to the adjoint, the incremental state and adjoint problems as illustrated in Section \ref{sec:Numerical} for a specific model. Then the eigenpairs of the covariance operator $\cC_1$ in \eqref{eq:PostCovariance} can be obtained by solving two generalized eigenvalue problems. 
By $\bbH_{\text{misfit}}$ we (formally) denote the matrix version of the Hessian misfit term $H_{\text{misfit}}$. We solve the first generalized eigenvalue problem with $\lambda_1 \geq \lambda_2 \geq \cdots $
\beq\label{eq:HmisfitEigenvalue}
\bbH_{\text{misfit}} \bspsi_j = \lambda_j \bbA_\alpha \bspsi_j, \quad j \geq 1, 
\eeq
by Lanczos or randomized method, which involves $O(J)$ PDE solves for $J$ eigenpairs. Then by algebraic manipulation we can write the posterior covariance matrix as \cite{spantini2015optimal}
\beq
(\bbH_{\text{misfit}} + \bbA_\alpha)^{-1} = \bbA^{-1}_\alpha - \Psi_J D_J \Psi_J^\top + O\left(\sum_{j> J} \frac{\lambda_j}{1+\lambda_j}\right),
\eeq
where $D_J = \text{diag}(\lambda_1/(1+\lambda_1), \dots, \lambda_J/(1+\lambda_J))$, with $J$ determined according to certain criteria, e.g., $\lambda_J$ is small enough, and $\Psi_J = (\bspsi_1, \dots, \bspsi_J)$. The second generalized eigenvalue problem is also solved by Lanczos or randomized method 
\beq\label{eq:postEigenvalue}
\bbM (\bbA^{-1}_\alpha - \Psi_J D_J \Psi_J^\top ) \bbM \bspsi_j^1 = \lambda_j^1 \bbM \bspsi_j^1, \quad j \geq 1,
\eeq
where $(\lambda_j^1, \bspsi_j^1)_{j\geq 1}$ are the eigenpairs corresponding to the covariance operator $\cC_1$. We remark that by solving the two step generalized eigenvalue problems, we don't need to assemble the Hessian matrix, which involves $O(N_h)$ PDE solves, but only need to perform $O(J)$ PDE solves. Considerable computational saving can thus be achieved for small $J$ when the eigenvalues of the Hessian of the misfit term and the posterior covariance decay fast.

\bibliographystyle{plain}
\bibliography{bibliography}

\begin{thebibliography}{10}

\bibitem{agapiou2015importance}
S~Agapiou, O~Papaspiliopoulos, D~Sanz-Alonso, and AM~Stuart.
\newblock Importance sampling: Computational complexity and intrinsic
  dimension.
\newblock {\em arXiv preprint arXiv:1511.06196}, 2015.

\bibitem{bachmayr2015sparse}
M.~Bachmayr, A.~Cohen, R.~DeVore, and G.~Migliorati.
\newblock Sparse polynomial approximation of parametric elliptic pdes. part ii:
  lognormal coefficients.
\newblock {\em arXiv preprint arXiv:1509.07050}, 2015.

\bibitem{bui2013computational}
T.~Bui-Thanh, O.~Ghattas, J.~Martin, and G.~Stadler.
\newblock A computational framework for infinite-dimensional bayesian inverse
  problems part i: The linearized case, with application to global seismic
  inversion.
\newblock {\em SIAM Journal on Scientific Computing}, 35(6):A2494--A2523, 2013.

\bibitem{calvetti1994implicitly}
D.~Calvetti, L.~Reichel, and D.C. Sorensen.
\newblock An implicitly restarted lanczos method for large symmetric eigenvalue
  problems.
\newblock {\em Electronic Transactions on Numerical Analysis}, 2(1):21, 1994.

\bibitem{chen2016adaptive}
P.~Chen.
\newblock Sparse quadrature for high-dimensional integration with {G}aussian
  measure.
\newblock {\em arXiv:1604.08466}, 2016.

\bibitem{chen2015new}
P.~Chen and A.~Quarteroni.
\newblock A new algorithm for high-dimensional uncertainty quantification based
  on dimension-adaptive sparse grid approximation and reduced basis methods.
\newblock {\em Journal of Computational Physics}, 298:176--193, 2015.

\bibitem{chen2015sparse}
P.~Chen and Ch. Schwab.
\newblock Sparse-grid, reduced-basis {B}ayesian inversion.
\newblock {\em Computer Methods in Applied Mechanics and Engineering}, 297:84
  -- 115, 2015.

\bibitem{chen2016sparse}
P~Chen and Ch. Schwab.
\newblock Sparse-grid, reduced-basis {B}ayesian inversion: Nonaffine-parametric
  nonlinear equations.
\newblock {\em Journal of Computational Physics}, 316:470 -- 503, 2016.

\bibitem{chkifa2014high}
A.~Chkifa, A.~Cohen, and Ch. Schwab.
\newblock High-dimensional adaptive sparse polynomial interpolation and
  applications to parametric pdes.
\newblock {\em Foundations of Computational Mathematics}, 14(4):601--633, 2014.

\bibitem{cohen2010convergence}
A.~Cohen, R.~DeVore, and C.~Schwab.
\newblock Convergence rates of best n-term galerkin approximations for a class
  of elliptic spdes.
\newblock {\em Foundations of Computational Mathematics}, 10(6):615--646, 2010.

\bibitem{cui2016dimension}
T.~Cui, K.J.H. Law, and Y.M. Marzouk.
\newblock Dimension-independent likelihood-informed mcmc.
\newblock {\em Journal of Computational Physics}, 304:109--137, 2016.

\bibitem{dashti2013map}
M.~Dashti, K.~Law, A.M. Stuart, and J.~Voss.
\newblock {MAP} estimators and their consistency in bayesian nonparametric
  inverse problems.
\newblock {\em Inverse Problems}, 29(9):095017, 2013.

\bibitem{dashti2013bayesian}
M.~Dashti and A.M. Stuart.
\newblock The {B}ayesian approach to inverse problems.
\newblock {\em arXiv preprint arXiv:1302.6989}, 2013.

\bibitem{ernst2016convergence}
O.G. Ernst, B~Sprungk, and L.~Tamellini.
\newblock Convergence of sparse collocation for functions of countably many
  {G}aussian random variables - with application to lognormal elliptic
  diffusion problems.
\newblock {\em arXiv:1611.07239}, 2016.

\bibitem{gerstner2003dimension}
T.~Gerstner and M.~Griebel.
\newblock Dimension--adaptive tensor--product quadrature.
\newblock {\em Computing}, 71(1):65--87, 2003.

\bibitem{girolami2011riemann}
M.~Girolami and B.~Calderhead.
\newblock Riemann manifold langevin and hamiltonian monte carlo methods.
\newblock {\em Journal of the Royal Statistical Society: Series B (Statistical
  Methodology)}, 73(2):123--214, 2011.

\bibitem{halko2011finding}
N.~Halko, P-G Martinsson, and J.A. Tropp.
\newblock Finding structure with randomness: Probabilistic algorithms for
  constructing approximate matrix decompositions.
\newblock {\em SIAM review}, 53(2):217--288, 2011.

\bibitem{hoang2013complexity}
V.H. Hoang, Ch. Schwab, and A.M. Stuart.
\newblock Complexity analysis of accelerated {MCMC} methods for {B}ayesian
  inversion.
\newblock {\em Inverse Problems}, 29(8):085010, 2013.

\bibitem{isaac2015scalable}
T.~Isaac, N.~Petra, G.~Stadler, and O.~Ghattas.
\newblock Scalable and efficient algorithms for the propagation of uncertainty
  from data through inference to prediction for large-scale problems, with
  application to flow of the {A}ntarctic ice sheet.
\newblock {\em Journal of Computational Physics}, 296:348--368, 2015.

\bibitem{klimke2006uncertainty}
A.~Klimke.
\newblock {\em Uncertainty modeling using fuzzy arithmetic and sparse grids.
  Universit{\"a}t Stuttgart}.
\newblock PhD thesis, Universit{\"a}t Stuttgart, Germany, 2006.

\bibitem{martin2012stochastic}
J.~Martin, L.C. Wilcox, C.~Burstedde, and O.~Ghattas.
\newblock A stochastic newton mcmc method for large-scale statistical inverse
  problems with application to seismic inversion.
\newblock {\em SIAM Journal on Scientific Computing}, 34(3):A1460--A1487, 2012.

\bibitem{marzouk2016introduction}
Y.~Marzouk, T.~Moselhy, M.~Parno, and A.~Spantini.
\newblock An introduction to sampling via measure transport.
\newblock {\em arXiv preprint arXiv:1602.05023}, 2016.

\bibitem{nobile2016adaptive}
F.~Nobile, L.~Tamellini, F.~Tesei, and R.~Tempone.
\newblock An adaptive sparse grid algorithm for elliptic {PDEs} with lognormal
  diffusion coefficient.
\newblock In {\em Sparse Grids and Applications-Stuttgart 2014}, pages
  191--220. Springer, 2016.

\bibitem{petra2014computational}
N.~Petra, J.~Martin, G.~Stadler, and O.~Ghattas.
\newblock A computational framework for infinite-dimensional bayesian inverse
  problems, part ii: Stochastic newton mcmc with application to ice sheet flow
  inverse problems.
\newblock {\em SIAM Journal on Scientific Computing}, 36(4):A1525--A1555, 2014.

\bibitem{reed1978methods}
M.~Reed and B.~Simon.
\newblock {\em Methods of Modern Mathematical Physics: Vol.: 4.: Analysis of
  Operators}.
\newblock Academic press, 1978.

\bibitem{Schillings2013}
C.~Schillings and Ch. Schwab.
\newblock Sparse, adaptive {S}molyak quadratures for {B}ayesian inverse
  problems.
\newblock {\em Inverse Problems}, 29(6), 2013.

\bibitem{schillings2014scaling}
C.~Schillings and Ch. Schwab.
\newblock Scaling limits in computational {B}ayesian inversion.
\newblock {\em ESAIM: Mathematical Modelling and Numerical Analysis},
  50(6):1825--1856, 2016.

\bibitem{schwab2012sparse}
Ch. Schwab and A.M. Stuart.
\newblock Sparse deterministic approximation of bayesian inverse problems.
\newblock {\em Inverse Problems}, 28(4):045003, 2012.

\bibitem{spantini2015optimal}
A.~Spantini, A.~Solonen, T.~Cui, J.~Martin, L.~Tenorio, and Y.~Marzouk.
\newblock Optimal low-rank approximations of bayesian linear inverse problems.
\newblock {\em SIAM Journal on Scientific Computing}, 37(6):A2451--A2487, 2015.

\bibitem{stuart2010inverse}
A.M. Stuart.
\newblock Inverse problems: a {B}ayesian perspective.
\newblock {\em Acta Numerica}, 19(1):451--559, 2010.

\bibitem{hippylib2016}
U.~Villa, N.~Petra, and O.~Ghattas.
\newblock h{IPPY}lib: {A}n {E}xtensible {S}oftware {F}ramework for
  {L}arge-{S}cale {D}eterministic and {L}inearized {B}ayesian {I}nverse
  {P}roblems.
\newblock Technical report, 2016.

\end{thebibliography}

\end{document}